%% file: torsionclass.tex
\documentclass[a4paper,12pt,centertags,oneside]{amsart}
\usepackage{amsmath,amstext,amsthm,a4,amssymb,amscd}
\usepackage[mathscr]{eucal}
\usepackage{mathrsfs}
\usepackage{bbold}
\usepackage{epsf}
\usepackage{typearea}
\usepackage{charter}
\usepackage{bm}

\usepackage{graphicx}
\usepackage[all]{xy}

\usepackage{hyperref}
\hypersetup{
    colorlinks = true,
    linkcolor = black,
    citecolor = black,
    urlcolor = black,
}

\newcommand{\Q}{\mathbb{Q}}
\newcommand{\R}{\mathbb{R}}
\newcommand{\C}{\mathbb{C}}
\newcommand{\Z}{\mathbb{Z}}
\newcommand{\N}{\mathbb{N}}

\newcommand{\CP}{\C\mathrm{P}}

\newcommand{\smooth}{{\mathscr{C}^\infty}}

\newcommand{\End}{\mathrm{End}}

\newcommand{\n}{\nabla}

\makeatletter
\newcommand*{\rom}[1]{\expandafter\@slowromancap\romannumeral #1@}
\makeatother

\DeclareMathOperator{\Id}{Id}
\DeclareMathOperator{\tr}{Tr}

\newtheorem{prop}{Proposition}[section]
\newtheorem{thm}[prop]{Theorem}

\newtheorem*{thmp}{Theorem 0.1'}

\theoremstyle{definition}
\newtheorem{defn}[prop]{Definition}

\theoremstyle{remark}

\numberwithin{equation}{section}
\title[Bismut-Lott torsion and Igusa-Klein torsion]{A comparison between the Bismut-Lott torsion \\ and the Igusa-Klein torsion}
\date{\today}

\author{Martin PUCHOL}
\address{Universit{\'e} Paris-Saclay,
CNRS,
Laboratoire de math{\'e}matiques d'Orsay,
F-91405 Orsay Cedex, France}
\email{martin.puchol@math.cnrs.fr}

\author{Yeping ZHANG}
\address{School of Mathematics,
Korea Institute for Advanced Study,
Hoegiro 85, Dongdaemungu,
Seoul 02455, Korea}
\email{ypzhang@kias.re.kr}

\author{Jialin ZHU}
\address{Mathematical Science Research Center,
Chongqing University of Technology,
No. 69 Hongguang Road,
Chongqing 400054, China}
\email{jialinzhu@cqut.edu.cn}

\begin{document}

\begin{abstract}
We consider a fibration with compact fiber together with a unitarily flat complex vector bundle over the total space.
Under the assumption that the fiberwise cohomology admits a filtration with unitary factors,
we construct Bismut-Lott analytic torsion classes.
The analytic torsion classes obtained satisfy Igusa's and Ohrt's axiomatization of higher torsion invariants.
As a consequence,
we obtain a higher version of the Cheeger-M{\"u}ller/Bismut-Zhang theorem:
for trivial flat line bundles,
the Bismut-Lott analytic torsion classes coincide with the Igusa-Klein higher topological torsions up to a normalization.
\end{abstract}

\maketitle

\tableofcontents

\section{Introduction}

We consider a unitarily flat complex vector bundle $(F,\n^F)$ over a closed manifold $X$ whose cohomology with coefficients in $F$ vanishes,
i.e., $H^\bullet(X,F)=0$.
Franz \cite{fr}, Reidemeister \cite{rei} and de Rham \cite{dr}
constructed a topological invariant associated with $(F,\n^F)$,
known as the Reidemeister-Franz topological torsion (RF-torsion).
RF-torsion is the first algebraic-topological invariant
which distinguishes certain homotopy-equivalent topological spaces \cite{fr,rei}.
RF-torsion could be extended to the case $H^\bullet(X,F)\neq 0$ \cite{dr,mil,wh}.
The construction of RF-torsion is based on the complex of simplicial chains in $X$ with values in $F$.

By replacing the complex of simplicial chains by the de Rham complex,
Ray and Singer \cite{rs}
obtained an analytic version of RF-torsion,
known as the Ray-Singer analytic torsion (RS-torsion).
In the same paper,
Ray and Singer conjectured that
RF-torsion and RS-torsion are equivalent.

Ray-Singer conjecture was proved independently by Cheeger \cite{c-cm} and M{\"u}ller \cite{m-cm}.
Their result is now known as the Cheeger-M{\"u}ller theorem.
Bismut, Zhang and M{\"u}ller simultaneously considered its extension.
M{\"u}ller \cite{m-cm-2} extended the Cheeger-M{\"u}ller theorem to the unimodular case,
i.e., the induced metric on the determinant line bundle $\det F$ is flat.
Bismut and Zhang \cite{bz} extended the Cheeger-M{\"u}ller theorem to arbitrary flat complex vector bundle.
There are also various extensions to equivariant cases
\cite{bz2,lr,lu}.

Wagoner \cite{wa} conjectured that
RF-torsion and RS-torsion can be extended to invariants of a fiber bundle,
i.e., a fibration $M\rightarrow S$
together with a flat complex vector bundle $(F,\n^F)$ over $M$.

Bismut and Lott \cite{bl} confirmed the analytic side of Wagoner's conjecture.
They extended RS-torsion to analytic torsion forms (BL-torsion).
Indeed,
Bismut and Lott proved a Riemann-Roch-Grothendieck type formula (RRG formula) for flat vector bundles.
The analytic torsion forms come as a refinement of their RRG formula at the level of differential forms.
Their work makes the analytic torsion into a consequence of the local index theory.

Inspired by the work of Bismut and Lott,
Igusa \cite{ig} confirmed the topological side of Wagoner's conjecture by constructing higher topological torsions,
known as the Igusa-Klein higher topological torsion (IK-torsion).
Goette, Igusa and Williams \cite{g-i-w,g-i} used IK-torsion to detect the exotic smooth structure of fiber bundles.
Dwyer, Weiss and Williams \cite{dww} constructed another version of higher topological torsion (DWW-torsion).

The relation among these higher torsions is a natural and important research topic.
We expect a higher version of the Cheeger-M{\"u}ller/Bismut-Zhang theorem in full generality.

By extending the proof of the Bismut-Zhang theorem \cite{bz},
Bismut and Goette \cite{bg} established a higher Cheeger-M{\"u}ller/Bismut-Zhang theorem
under the assumption that
there exist a fiberwise Morse function $f: M \rightarrow \R$ and a fiberwise Riemannian metric
such that the fiberwise gradient of $f$ is Morse-Smale \cite{sm}.
Goette \cite{g-a1,g-a2} extended the results in \cite{bg} to arbitrary fiberwise Morse functions.
Bismut and Goette \cite{bg} also extended BL-torsion to the equivariant case.
And there are related works \cite{bg2,bu}.
We refer to the survey by Goette \cite{g} for an overview on higher torsion invariants.
Goette also proposed a program extending the argument in \cite{g-a1,g-a2}
to functions with both non-degenerate critical points and birth-death critical points.

An alternative approach to the higher Cheeger-M{\"u}ller/Bismut-Zhang theorem is based on Igusa's work \cite{ig2}.
Igusa axiomatized higher torsion invariants of fibrations (equipped with trivial flat complex line bundles).
He considered a fibration $M \rightarrow S$ with closed oriented fiber $Z$ such that $H^\bullet(Z)$ is unipotent.
He stated two axioms,
called the additivity axiom and the transfer axiom,
and showed that an invariant of $M/S$ satisfies the axioms if and only if
it is a linear combination of IK-torsion and the higher Miller-Morita-Mumford class \cite{miller,mor,mum} of $M/S$.
Badzioch, Dorabiala, Klein and Williams \cite{bdkw} showed that DWW-torsion satisfies Igusa's axioms.

One of the key results in this paper states that BL-torsion satisfies Igusa's axioms.
As a consequence,
in Theorem \ref{thm-bg},
we establish a higher Cheeger-M{\"u}ller/Bismut-Zhang theorem for fibrations equipped with trivial flat complex line bundles.
The proof is based on the results of Ma \cite{ma} and ours \cite{pzz2}.
Ma's result \cite{ma},
which describes the behavior of BL-torsion under the composition of submersions,
shows that BL-torsion satisfies the transfer axiom.
Our result \cite{pzz2},
which describes the behavior of BL-torsion under gluing,
shows that BL-torsion satisfies the additivity axiom.
Our result \cite{pzz2} is a higher version of the gluing formula obtained by Br{\"u}ning and Ma \cite{bm}.
And there are related results \cite{lu,pzz,vi,imrn-zhu,israel-zhu}.
In particular,
in \cite{pzz},
we gave a purely analytic proof of \cite{bm}.
The technique applied in \cite{pzz} is closely related to \cite{pzz2}.

Ohrt \cite{ohrt} axiomatized higher torsion invariants of certain fibrations equipped with unitarily flat complex vector bundles.
He considered a smooth manifold $S$ whose fundamental group $\pi_1(S)$ is finite,
a fibration $M \rightarrow S$ with simple closed oriented fiber $Z$,
and a unitarily flat complex vector bundle $F$ over $M$ with finite holonomy group such that $H^\bullet(Z,F)$ is unipotent.
He showed that an invariant of $(M/S,F)$ satisfies his axioms
if and only if it is a linear combination of IK-torsion and the higher Miller-Morita-Mumford class of $(M/S,F)$.

In this paper,
we also show that BL-torsion satisfies Ohrt's axioms.
Moreover,
under the same assumptions as in \cite{ohrt},
we establish a higher Cheeger-M{\"u}ller/Bismut-Zhang theorem.

Let us now give more details about the matter of this paper.

\vspace{5mm}

\noindent\textbf{Odd characteristic form and torsion form.}
Let $M$ be a smooth manifold.
Let $(F,\n^F)$ be a flat complex vector bundle over $M$ with flat connection $\n^F$.
Let $g^F$ be a Hermitian metric on $F$.
Let $\overline{F}^*$ be the bundle of antilinear functionals on $F$.
We will view $g^F$ as a map from $F$ to $\overline{F}^*$.
Following \cite[(4.1)]{bz} and \cite[(1.31)]{bl},
we define
\begin{equation}
\label{intro-def-omegaF}
\omega(F,g^F) = \big(g^F\big)^{-1} \n^F g^F
\in \Omega^1(M,\mathrm{End}(F)) \;.
\end{equation}
We fix a square root of $i$,
denoted by $i^{1/2}$.
In what follows,
the choice of square root will be irrelevant.
Let $\varphi: \Omega^\bullet(M) \rightarrow \Omega^\bullet(M)$ be such that
\begin{equation}
\label{eq-def-varphi}
\varphi \omega = (2\pi i)^{-k/2}\omega \hspace{2.5mm} \text{for }\omega\in\Omega^k(M) \;.
\end{equation}
Let $f$ be an odd polynomial,
i.e., $f(-z)=-f(z)$.
Following \cite[(1.34)]{bl},
we define
\begin{equation}
f\big(\n^F,g^F\big) = (2\pi i)^{1/2} \varphi \tr\Big[f\Big(\frac{\omega(F,g^F)}{2}\Big)\Big]
\in \Omega^\mathrm{odd}(M) \;.
\end{equation}
Bismut and Lott \cite[\textsection I]{bl} showed that $f\big(\n^F,g^F\big)$ is a closed real form
and its cohomology class
\begin{equation}
f\big(\n^F\big) := \big[f\big(\n^F,g^F\big)\big]\in H^\mathrm{odd}(M)
\end{equation}
is independent of $g^F$.
For a graded flat complex vector bundle
$\big(F^\bullet = \bigoplus_k F^k,\n^{F^\bullet} = \bigoplus_k \n^{F^k}\big)$
and a Hermitian metric $g^{F^\bullet} = \bigoplus_k g^{F^k}$ on $F^\bullet$,
we denote
\begin{align}
\begin{split}
\label{intro-fsum}
f\big(\n^{F^\bullet},g^{F^\bullet}\big) & = \sum_k (-1)^k f\big(\n^{F^k},g^{F^k}\big)
\in \Omega^\mathrm{odd}(M) \;,\\
f\big(\n^{F^\bullet}\big) & = \sum_k (-1)^k f\big(\n^{F^k}\big)
\in H^\mathrm{odd}(M) \;.
\end{split}
\end{align}
If $f$ is an odd formal power series,
the constructions above still make sense.
In this paper,
we always take
\begin{equation}
f(z) = ze^{z^2} \;.
\end{equation}

Now let
\begin{equation}
\label{intro-E}
\big(E^\bullet,\n^{E^\bullet},\partial\big) : \; 0 \rightarrow E^0 \rightarrow \cdots \rightarrow E^n \rightarrow 0
\end{equation}
be an exact sequence of flat complex vector bundles over $M$.
More precisely,
$(E^\bullet,\partial)$ is an exact sequence of complex vector bundles over $M$,
and $\partial$ commutes with the flat connection $\n^{E^\bullet}$.
By \cite[Thm. 2.19]{bl},
we have
\begin{equation}
\label{intro-f0}
f\big(\n^{E^\bullet}\big) = 0 \in H^\mathrm{odd}(M) \;.
\end{equation}
Let $g^{E^\bullet} = \bigoplus_k g^{E^k}$ be a Hermitian metric on $E^\bullet$.
The torsion form \cite[Def. 2.20]{bl} is a real even differential form $\mathscr{T}\big(\n^{E^\bullet},\partial,g^{E^\bullet}\big)$ on $M$ satisfying
\begin{equation}
\label{intro-tf}
d \mathscr{T}\big(\n^{E^\bullet},\partial,g^{E^\bullet}\big) = f\big(\n^{E^\bullet}, g^{E^\bullet}\big) \;.
\end{equation}
We remark that \eqref{intro-f0} and \eqref{intro-tf} can be extended to the case where \eqref{intro-E} is not necessarily exact \cite[\textsection II]{bl}.

\vspace{5mm}

\noindent\textbf{R.R.G. formula and analytic torsion form.}
Let
\begin{equation}
\pi: M \rightarrow S
\end{equation}
be a fibration with closed fiber $Z$.
Let $o(TZ)$ be the orientation line of the fiberwise tangent bundle $TZ$.
Let
\begin{equation}
e(TZ)\in H^{\dim Z}\big(M,o(TZ)\big)
\end{equation}
be the Euler class of $TZ$ (see \cite[(3.17)]{bz}).

Let
\begin{equation}
(F,\n^F)
\end{equation}
be a flat complex vector bundle over $M$.
Let $H^\bullet(Z,F)$ be the fiberwise cohomology of $Z$ with coefficients in $F$.
Then $H^\bullet(Z,F)$ is a graded complex vector bundle over $S$
equipped with a canonical flat connection $\n^{H^\bullet(Z,F)}$ (see \cite[Def. 2.4]{bl}).
Bismut and Lott \cite[Thm. 3.17]{bl} established the following Riemann-Roch-Grothendieck type formula
\begin{equation}
\label{intro-RRG}
f\big(\n^{H^\bullet(Z,F)}\big) = \int_Z e(TZ) f\big(\n^F\big) \in H^\mathrm{odd}(S) \;.
\end{equation}

Now let $T^HM\subseteq TM$ be a complement of $TZ$,
let $g^{TZ}$ be a Riemannian metric on $TZ$,
let $g^F$ be a Hermitian metric on $F$.
Let $\n^{TZ}$ be the Bismut connection \cite[Def. 1.6]{b} (see also \cite[Thm. 1.1]{b97} and \cite[Prop. 10.2]{bgv}) on $TZ$ associated with $T^HM$ and $g^{TZ}$.
Let
\begin{equation}
e\big(TZ,\n^{TZ}\big) \in \Omega^{\dim Z}\big(M,o(TZ)\big)
\end{equation}
be the Euler form (see \cite[(3.17)]{bz}).
Let $g^{H^\bullet(Z,F)}$ be the $L^2$-metric on $H^\bullet(Z,F)$ associated with $g^{TZ},g^F$.
Bismut and Lott \cite[Def. 3.22]{bl} constructed a real even differential form $\mathscr{T}\big(T^HM,g^{TZ},g^F\big)$ on $S$ such that
\begin{equation}
\label{intro-dT}
d\mathscr{T}\big(T^HM,g^{TZ},g^F\big) =
\int_Z e\big(TZ,\n^{TZ}\big) f\big(\n^F,g^F\big) - f\big(\n^{H^\bullet(Z,F)},g^{H^\bullet(Z,F)}\big) \;.
\end{equation}
We call $\mathscr{T}\big(T^HM,g^{TZ},g^F\big)$ the Bismut-Lott analytic torsion form.
Zhu \cite[\textsection 2]{imrn-zhu} extended the R.R.G. formula \eqref{intro-RRG} as well as the analytic torsion form to the case
where $Z$ is a compact manifold with boundaries and $T^HM,g^{TZ},g^F$ are product on a tubular neighborhood of the boundary (see \cite[(2.30)-(2.35)]{imrn-zhu}).

\vspace{5mm}

\noindent\textbf{Torsion classes.}
For a smooth manifold $S$,
there is a canonical bijection
\begin{align}
\begin{split}
& \big\{ \text{flat complex vector bundles over } S \big\}/\text{isomorphism} \\
& \simeq \big\{ \text{linear representations of } \pi_1(S) \big\}/\text{conjugation} \;.
\end{split}
\end{align}
Let $(E,\n^E)$ be flat complex vector bundle over $S$.
If $(E,\n^E)$ corresponds to a unitary representation of $\pi_1(S)$,
we call $(E,\n^E)$ a unitarily flat complex vector bundle,
and we sometimes simply say that $(E,\n^E)$ is unitary.
A flat complex vector bundle is unitary if and only if it admits a flat Hermitian metric.
For a unitarily flat complex vector bundle $(E,\n^E)$ and two flat Hermitian metrics $g^E,{g^E}'$ on $E$,
we can find an automorphism of flat complex vector bundle $\phi: E \rightarrow E$ such that ${g^E}' = \phi^*g^E$.

Let $(E,\n^E)$ be a flat complex vector bundle.
If there is a filtration by flat subbundles
\begin{equation}
E = E_r \supseteq E_{r-1} \supseteq \cdots \supseteq E_0 = 0
\end{equation}
such that $E_j/E_{j-1}$,
equipped with the flat connection induced by $\n^E$,
is unitary for each $j$,
we say that $(E,\n^E)$ is filtered by flat subbundles with unitary factors.
Moreover,
if $E_j/E_{j-1}$ is a trivial flat complex line bundle for each $j$,
we say that $(E,\n^E)$ is unipotent.

\textbf{From now on,
we always assume that $(F,\n^F)$ is unitary,
and $H^\bullet(Z,F)$ is filtered by flat subbundles with unitary factors.}
Let $g^F$ be a flat Hermitian metric on $F$.
In this paper,
we introduce a closed form (see \eqref{eq-def-tcl} and \eqref{eq-def-dtcl})
\begin{equation}
\label{intro-Tcl}
\mathscr{T}_\mathrm{cl}\big(T^HM,g^{TZ},g^F\big)
= \mathscr{T}^{[>0]}\big(T^HM,g^{TZ},g^F\big) + \text{correction terms} \in \Omega^{\mathrm{even}\geqslant 2}(S) \;,
\end{equation}
where $\mathscr{T}^{[>0]}\big(T^HM,g^{TZ},g^F\big)$ consists of the components of $\mathscr{T}\big(T^HM,g^{TZ},g^F\big)$ of positive degree.
The correction terms are torsion forms of certain exact sequences induced by a filtration of $H^\bullet(Z,F)$ with unitary factors,
which were introduced by Ma \cite[Def. 3.1]{ma}.
The Bismut-Lott analytic torsion class (see Definition \ref{def-tclass}) is defined as
\begin{equation}
\label{intro-taubl}
\tau^\mathrm{BL}(M/S,F) := \Big[ \mathscr{T}_\mathrm{cl}\big(T^HM,g^{TZ},g^F\big) \Big] \in H^{\mathrm{even}\geqslant 2}(S) \;.
\end{equation}
We will see that $\tau^\mathrm{BL}(M/S,F)$ is uniquely determined by the fibration $\pi: M \rightarrow S$ and the flat complex vector bundle $(F,\n^F)$.
One of the many reasons to drop the component of degree zero in \eqref{intro-Tcl} is to make $\tau^\mathrm{BL}(M/S,F)$ independent of $g^F$.
Goette \cite[Def. 2.8]{g} constructed $\tau^\mathrm{BL}(M/S,F)$ for $H^\bullet(Z,F)$ unitary.

\textbf{We further assume that $H^\bullet(Z,F)$ is unipotent,
and the fiber of $\pi: M \rightarrow S$ is oriented,
i.e., the fiberwise tangent bundle $TZ$ is oriented.}
Igusa \cite{ig} constructed a higher topological torsion
\begin{equation}
\label{intro-tauig}
\tau^\mathrm{IK}(M/S,F) \in H^\mathrm{even}(S) \;,
\end{equation}
which we call the Igusa-Klein higher topological torsion.

\vspace{5mm}

\noindent\textbf{Higher Cheeger-M{\"u}ller/Bismut-Zhang theorem for trivial flat complex line bundles.}
Let $\pi: M\rightarrow S$ be a smooth fibration with closed oriented fiber.
Let $\mathbb{1}$ be the trivial flat complex line bundle over $M$.

We assume that $H^\bullet(Z) = H^\bullet(Z,\mathbb{1})$ is unipotent.

We denote by
\begin{equation}
\label{intro-k}
\tau^\mathrm{BL}_k(M/S,\mathbb{1}), \; \tau^\mathrm{IK}_k(M/S,\mathbb{1}) \in H^{2k}(S)
\end{equation}
the components of $\tau^\mathrm{BL}(M/S,\mathbb{1})$ and $\tau^\mathrm{IK}(M/S,\mathbb{1})$ of degree $2k$.

Let $\zeta(\cdot)$ be the Riemann zeta function.

The following theorem confirms the higher Cheeger-M{\"u}ller/Bismut-Zhang theorem proposed by Goette \cite[Thm. 5.5]{g} with $F = \mathbb{1}$.

\begin{thm}
\label{thm-bg}
For any positive integer $k$
and any smooth fibration $\pi: M\rightarrow S$ with closed oriented fiber $Z$ such that $H^\bullet(Z)$ is unipotent,
we have
\begin{align}
\label{eq-thm-bg}
\begin{split}
& \frac{2^{4k}\big((2k)!\big)^2}{(4k+1)!} \tau^\mathrm{BL}_{2k}(M/S,\mathbb{1}) \\
& = - \frac{(2k)!}{(2\pi)^{2k}} \tau^\mathrm{IK}_{2k}(M/S,\mathbb{1}) + \frac{\zeta'(-2k)}{2} \Big[\int_Z e(TZ)\mathrm{ch}(TZ) \Big]^{[4k]}  \in H^{4k}(S) \;.
\end{split}
\end{align}
\end{thm}

Here the coefficient on the left hand side of \eqref{eq-thm-bg} is exactly the Chern normalization introduced by Bismut and Goette \cite[Def. 2.37]{bg}.
And the second term on the right hand side of \eqref{eq-thm-bg} is a special case of the characteristic class $^0\!J(\cdot,\cdot)$ constructed by Bismut and Goette \cite[Def. 7.3]{bg}
(see also \cite[(3.20)]{g}).
In \cite[\textsection 3.3]{ig2},
this term is called the higher Miller-Morita-Mumford class \cite{miller,mor,mum}.

We remark that
\begin{equation}
\tau^\mathrm{BL}_{2k+1}(M/S,\mathbb{1}) = 0 \in H^{4k+2}(S) \;.
\end{equation}
This is due to Bismut and Lott \cite{bl} (see also Theorem \ref{thm-trivial}).

By \cite[Cor. 4.8]{ig2},
if the fiber of $M/S$ is even-dimensional,
then
\begin{equation}
\label{intro-bg0}
\tau^\mathrm{IK}_{2k}(M/S,\mathbb{1}) =  \frac{(-1)^k\zeta(2k+1)}{4}\Big[\int_Z e(TZ)\mathrm{ch}(TZ) \Big]^{[4k]} \;.
\end{equation}
By \eqref{intro-bg0} and the fact that
\begin{equation}
\zeta(2k+1) = \frac{(-1)^k2^{2k+1}\pi^{2k}}{(2k)!}\zeta'(-2k) \;,
\end{equation}
Theorem \ref{thm-bg} is equivalent to the following theorem.

\begin{thmp}
Under the same assumptions as in Theorem \ref{thm-bg},
if $Z$ is odd-dimensional,
then
\begin{equation}
\frac{2^{4k}\big((2k)!\big)^2}{(4k+1)!} \tau^\mathrm{BL}_{2k}(M/S,\mathbb{1})
= - \frac{(2k)!}{(2\pi)^{2k}} \tau^\mathrm{IK}_{2k}(M/S,\mathbb{1}) \in H^{4k}(S) \;,
\end{equation}
if $Z$ is even-dimensional,
then
\begin{equation}
\tau^\mathrm{BL}_{2k}(M/S,\mathbb{1})= 0 \in H^{4k}(S) \;.
\end{equation}
\end{thmp}

Now we briefly explain the proof of Theorem \ref{thm-bg}'.

One of the key ingredients in the proof is Igusa's axiomatization of higher torsion invariants \cite{ig2}.
We consider an invariant $\tau$ assigning a cohomology class
\begin{equation}
\tau(M/S)\in H^\mathrm{even}(S)
\end{equation}
to any smooth fibration $\pi: M\rightarrow S$ with closed oriented fiber $Z$ such that $H^\bullet(Z)$ is unipotent.
Similarly to \eqref{intro-k},
we denote by $\tau_k(M/S) \in H^{2k}(S)$ the components of $\tau(M/S)$ of degree $2k$.
By \cite[Cor. 4.5]{ig2},
if $\tau$ satisfies Axioms 1-3 in \textsection \ref{subsect-axiom} with $F=\mathbb{1}$,
then there exist $a_{2k},b_{2k}\in\R$ such that for any $M/S$ under consideration,
we have
\begin{align}
\label{intro-ig}
\begin{split}
\tau_{2k}(M/S) & = a_{2k} \tau^\mathrm{IK}_{2k}(M/S,\mathbb{1}) \hspace{5mm} \text{if } Z \text{ is odd-dimensional} \;, \\
\tau_{2k}(M/S) & = b_{2k} \Big[\int_Z e(TZ)\mathrm{ch}(TZ) \Big]^{[4k]} \hspace{5mm} \text{if } Z \text{ is even-dimensional} \;.
\end{split}
\end{align}

Axiom 1 trivially holds for $\tau^\mathrm{BL}$.
We will show that Axiom 2 holds for $\tau^\mathrm{BL}$ using our recent result \cite{pzz2}.
And we will show that Axiom 3 holds for $\tau^\mathrm{BL}$ using the result of Ma \cite{ma}.
Here the calculation is quite straightforward.

Now we know that for certain $a_{2k},b_{2k}\in\R$,
the identities in \eqref{intro-ig} hold with $\tau(M/S)$ replaced by $\tau^\mathrm{BL}_{2k}\big(M/S,\mathbb{1}\big)$.
To find $a_{2k},b_{2k}$,
it is sufficient to consider $\mathbb{S}^n$-bundles with $n=1,2$.

\vspace{5mm}

\noindent\textbf{Partial result for unitarily flat complex vector bundles.}
A topological space $X$ is called simple,
if $X$ is connected, the fundamental group $\pi_1(X)$ is abelian, and $\pi_1(X)$ acts trivially on the higher homotopy groups $\pi_{\geqslant 2}(X)$.

Let $S$ be smooth manifold with $\pi_1(S)$ finite.
Let $\pi: M\rightarrow S$ be a smooth fibration with simple closed oriented fiber.
Let $(F,\n^F)$ be a unitarily flat complex vector bundle over $M$ with finite holonomy group.

We assume that $H^\bullet(Z,F)$ is unipotent.

Similarly to \eqref{intro-k},
we denote by $\tau^\mathrm{BL}_k(M/S,F), \tau^\mathrm{IK}_k(M/S,F) \in H^{2k}(S)$
the components of $\tau^\mathrm{BL}(M/S,F), \tau^\mathrm{IK}(M/S,F)$ of degree $2k$.

The following theorem partially confirms the higher Cheeger-M{\"u}ller/Bismut-Zhang theorem proposed by Goette \cite[Thm. 5.5]{g}.

\begin{thm}
\label{thm-bg2}
Under the assumptions above,
for any positive integer $k$,
we have
\begin{align}
\label{eq-thm-bg2}
\begin{split}
& \frac{2^{4k}\big((2k)!\big)^2}{(4k+1)!}\tau^\mathrm{BL}_{2k}(M/S,F) \\
& = - \frac{(2k)!}{(2\pi)^{2k}} \tau^\mathrm{IK}_{2k}(M/S,F) + \frac{\zeta'(-2k)\mathrm{rk}F}{2} \Big[\int_Z e(TZ)\mathrm{ch}(TZ) \Big]^{[4k]}  \in H^{4k}(S) \;,
\end{split}
\end{align}
for any non negative integer $k$,
we have
\begin{equation}
\label{eq2-thm-bg2}
\frac{2^{4k+2}\big((2k+1)!\big)^2}{(4k+3)!}\tau^\mathrm{BL}_{2k+1}(M/S,F)
= - \frac{(2k+1)!}{(2\pi)^{2k+1}} \tau^\mathrm{IK}_{2k+1}(M/S,F)  \in H^{4k+2}(S) \;.
\end{equation}
\end{thm}

One of the key ingredients in the proof of Theorem \ref{thm-bg2} is Ohrt's axiomatization of higher torsion invariants \cite{ohrt}.
We will show that $\tau^\mathrm{BL}$ satisfies Ohrt's axioms (see Axioms 1-7 in \textsection \ref{subsect-axiom}).
Then,
by \cite[Theorem 0.1]{ohrt},
there exist $a_{2k},b_{2k}\in\R$ such that for any $(M/S,F)$ under consideration,
we have
\begin{equation}
\frac{2^{4k}\big((2k)!\big)^2}{(4k+1)!} \tau^\mathrm{BL}_{2k}(M/S,F)
= a_{2k} \tau^\mathrm{IK}_{2k}(M/S,F) + b_{2k} \mathrm{rk}F \Big[\int_Z e(TZ)\mathrm{ch}(TZ)\Big]^{[4k]} \;.
\end{equation}
We can find $a_{2k},b_{2k}$ by considering $\mathbb{S}^n$-bundles equipped with trivial flat complex line bundles.
And \eqref{eq2-thm-bg2} follows from a similar argument.

\vspace{5mm}

\noindent\textbf{Notations.}
We summarize some frequently used notations and conventions.

For a smooth manifold $S$,
we denote by $\Omega^k(S)$ the vector space of differential forms on $S$ of degree $k$.
In particular,
we have $\Omega^0(S) = \smooth(S)$.
We denote
\begin{equation}
\label{notation-Omega}
\Omega^\bullet(S) = \bigoplus_k \Omega^k(S) \;,\hspace{5mm}
\Omega^\mathrm{even/odd}(S) = \bigoplus_{k \; \mathrm{even/odd}} \Omega^k(S) \;.
\end{equation}
For $\omega = \sum_k \omega_k \in \Omega^\bullet(S)$ with $\omega_k \in \Omega^k(S)$,
we denote
\begin{equation}
\label{intro-deg}
\omega^{[k]} = \omega_k \;,\hspace{5mm}
\omega^{[>k]} = \sum_{j>k}\omega_j \;.
\end{equation}
Following \cite[Def. 1.10]{bl},
we denote by
\begin{equation}
\label{intro-Q}
Q^S \subseteq \Omega^\mathrm{even}(S)
\end{equation}
the vector subspace of real even forms on $S$,
and denote by
\begin{equation}
\label{intro-Q0}
Q^{S,0}\subseteq Q^S
\end{equation}
the vector subspace of exact real even forms on $S$.

For a vector bundle $E$ over $S$,
we denote by $\Omega^k(S,E)$ the vector space of differential forms on $S$ of degree $k$ with values in $E$.
In particular,
we have $\Omega^0(S,E) = \smooth(S,E)$.
We use the notation $\Omega^\bullet(S,E)$ in the same way as in \eqref{notation-Omega}.

\vspace{5mm}

\noindent\textbf{Acknowledgments.}
The authors are grateful to Professor Xiaonan Ma for having raised the question which is solved in this paper.
The authors thank Professor Bismut and Professor Igusa for their interest in this work.

Zhang is supported by KIAS individual Grant MG077401 at Korea Institute for Advanced Study.

\section{Preliminaries}

\subsection{Torsion form}
This subsection follows \cite[\textsection II]{bl}.

Let $S$ be a smooth manifold.
Let
\begin{equation}
\big(E^\bullet,\n^{E^\bullet},\partial\big) : \; 0 \rightarrow E^0 \rightarrow \cdots \rightarrow E^n \rightarrow 0
\end{equation}
be an exact sequence of flat complex vector bundles over $S$.
We extend the action $\partial: \smooth(S,E^i) \rightarrow \smooth(S,E^{i+1})$ to $\partial: \Omega^k(S,E^i) \rightarrow \Omega^k(S,E^{i+1})$ such that
\begin{equation}
\partial (\tau \otimes w) = (-1)^k \tau \otimes \partial w \hspace{2.5mm}
\text{for } \tau\in\Omega^k(S) \text{ and } w\in\smooth(S,E^i) \;.
\end{equation}
We extend the action $\n^{E^\bullet}: \smooth(S,E^\bullet) \rightarrow \Omega^1(S,E^\bullet)$ to $\n^{E^\bullet}: \Omega^k(S,E^\bullet) \rightarrow \Omega^{k+1}(S,E^\bullet)$ such that
\begin{equation}
\n^{E^\bullet}(\tau \otimes w) = d\tau \otimes w +(-1)^k \tau \wedge \n^{E^\bullet} w \hspace{2.5mm}
\text{for } \tau\in\Omega^k(S) \text{ and } w\in\smooth(S,E^\bullet) \;.
\end{equation}
Under the assumptions and the conventions above,
we have
\begin{equation}
\label{eq-nW-sq}
\partial^2 = 0 \;,\hspace{5mm}
\big(\n^{E^\bullet}\big)^2 = 0 \;,\hspace{5mm}
\partial \n^{E^\bullet} + \n^{E^\bullet} \partial = 0 \;.
\end{equation}
Set
\begin{equation}
\label{eq-Aprimprim}
A'' = \partial + \n^{E^\bullet} \;.
\end{equation}
By \eqref{eq-nW-sq} and \eqref{eq-Aprimprim},
we have
\begin{equation}
\big( A'' \big)^2 = 0 \;.
\end{equation}
Thus $A''$ is a flat superconnection in the sense of \cite[\textsection I]{bl}.

Let $g^{E^\bullet} = \bigoplus_{k=0}^n g^{E^k}$ be a Hermitian metric on $E^\bullet$.
Similarly to \eqref{intro-def-omegaF},
we define
\begin{equation}
\label{eq-def-omegaW}
\omega^{E^\bullet} =
\big(g^{E^\bullet}\big)^{-1} \n^{E^\bullet} g^{E^\bullet} \in \Omega^1\big(S,\End(E^\bullet)\big) \;.
\end{equation}
Let $\partial^*$ be the adjoint of $\partial$.
Let $A'$ be the adjoint superconnection of $A''$ in the sense of \cite[\textsection I]{bl}.
By \cite[\textsection II(b)]{bl},
we have
\begin{equation}
A' = \partial^* + \n^{E^\bullet} + \omega^{E^\bullet} \;.
\end{equation}
Set
\begin{equation}
X
= \frac{1}{2} (A' - A'')
= \frac{1}{2}(\partial^*-\partial) + \frac{1}{2} \omega^{E^\bullet}
\in \Omega^\bullet\big(S,\End(E^\bullet)\big) \;.
\end{equation}
Let $N^{E^\bullet}$ be the number operator on $E^\bullet$,
i.e., $N^{E^\bullet}\big|_{E^k} = k\Id$.
For $t>0$,
let $X_t$ be the operator $X$ with the metric $g^{E^\bullet}$ replaced by $t^{N^{E^\bullet}}g^{E^\bullet}$.
We have
\begin{equation}
\label{eq-def-Xt}
X_t = \frac{1}{2}(t \partial^*-\partial) + \frac{1}{2} \omega^{E^\bullet} \;.
\end{equation}

We denote
\begin{equation}
\chi'(E^\bullet) = \sum_{k=0}^n (-1)^kk \,\mathrm{rk}\big(E^k\big) \;.
\end{equation}

Recall that $f(z) = z e^{z^2}$ and $f'(z) = (1+2z^2) e^{z^2}$.
Recall that $Q^S$ was defined in \eqref{intro-Q}.

The following definition is due to Bismut and Lott \cite[Def. 2.20]{bl}.

\begin{defn}
\label{def-tf}
The torsion form associated with
$(\n^{E^\bullet},\partial,g^{E^\bullet})$ is defined as
\begin{align}
\label{eq-def-torsion-form}
\begin{split}
& \mathscr{T}\big(\n^{E^\bullet},\partial,g^{E^\bullet}\big) \\
& =  - \int_0^{+\infty} \bigg\{ \varphi \tr \Big[(-1)^{N^{E^\bullet}}\frac{N^{E^\bullet}}{2} f'(X_t) \Big]
- \frac{1}{2} \chi'(E^\bullet) f'\Big(\frac{i\sqrt{t}}{2}\Big) \bigg\} \frac{dt}{t}
\in Q^S \;.
\end{split}
\end{align}
By \cite[Thm. 2.13, Prop. 2.18]{bl},
the integrand in \eqref{eq-def-torsion-form} is integrable.
\end{defn}

Let $f\big(\n^{E^\bullet},g^{E^\bullet}\big) \in \Omega^\mathrm{odd}(S)$
be as in \eqref{intro-fsum}.
By \cite[Thm. 2.22]{bl},
we have
\begin{equation}
\label{eq-dtf}
d \mathscr{T}\big(\n^{E^\bullet},\partial,g^{E^\bullet}\big) = f\big(\n^{E^\bullet},g^{E^\bullet}\big) \;.
\end{equation}

The following theorem is due to Bismut and Lott \cite[Thm. 1.8(iv)]{bl}.

We will use the notation in \eqref{intro-deg}.

\begin{thm}
\label{thm-tf-4kplus2}
If the quadruple $(E^\bullet,\n^{E^\bullet},\partial,g^{E^\bullet})$
is the complexification of a real quadruple $(E^\bullet_\R,\n^{E^\bullet_\R},\partial_\R,g^{E^\bullet_\R})$,
then
\begin{equation}
\label{eq-thm-tf-4kplus2}
\mathscr{T}^{[k]}\big(\n^{E^\bullet},\partial,g^{E^\bullet}\big) = 0 \hspace{2.5mm} \text{for } k \equiv 2 \; (\mathrm{mod} \; 4) \;.
\end{equation}
\end{thm}
\begin{proof}
Note that $N^{E^\bullet}$ is self-adjoint,
the proof of \cite[Thm. 1.8(iv)]{bl} yields
\begin{equation}
\overline{\tr \left[(-1)^{N^{E^\bullet}} \frac{N^{E^\bullet}}{2} f'(X_t) \right]^{[k]}}
= (-1)^{k(k-1)/2} \tr \left[(-1)^{N^{E^\bullet}} \frac{N^{E^\bullet}}{2} f'(X_t) \right]^{[k]} \;.
\end{equation}
From Definition \ref{def-tf} and the fact that $\mathscr{T}\big(\n^{E^\bullet},\partial,g^{E^\bullet}\big)$ is real,
we obtain \eqref{eq-thm-tf-4kplus2}.
This completes the proof.
\end{proof}

By the end of this subsection,
we introduce several notations.

For a flat complex vector bundle $(E,\n^E)$ over $S$
and Hermitian metrics $g^E,{g^E}'$ on $E$,
we denote by $\mathscr{T}\big(g^E,{g^E}'\big)$ the torsion form of the exact sequence
\begin{equation}
\label{eq-EE}
0 \rightarrow E \xrightarrow{\mathrm{Id}} E \rightarrow 0 \;,
\end{equation}
where the first $E$ is equipped with the metric $g^E$
and the second $E$ is equipped with the metric ${g^E}'$.

For a flat complex vector bundle $(E,\n^E)$ over $S$,
a flat subbundle $F\subseteq E$
and Hermitian metrics $g^F,g^E,g^{E/F}$ on $F,E,E/F$,
we denote by $\mathscr{T}\big(g^F,g^E,g^{E/F}\big)$ the torsion form of the exact sequence
\begin{equation}
\label{eq-FE}
0 \rightarrow F \rightarrow E \rightarrow E/F \rightarrow 0
\end{equation}
equipped with metrics $g^F,g^E,g^{E/F}$.

\subsection{Torsion form of filtration with unitary factors}
\label{subsect-lim-tf}

Let $(E,\n^E)$ be a flat complex vector bundle over $S$.
We assume that $E$ is filtered by flat subbundles with unitary factors.
More precisely,
there is a filtration of $E$ by flat subbundles
\begin{equation}
E_\bullet: \; E = E_r \supseteq E_{r-1} \supseteq \cdots \supseteq E_1 \supseteq E_0 = 0 \;,
\end{equation}
such that $F_j := E_j/E_{j-1}$ admits a flat Hermitian metric for each $j$.

Let $g^E$ be a Hermitian metric on $E$.
Let $g^{E_j}$ be the restricted metric on $E_j$.
Let $g^{F_j}$ be a flat Hermitian metric on $F_j$.
Recall that $\mathscr{T}(\cdot,\cdot,\cdot)$ was defined in the paragraph containing \eqref{eq-FE}.
Following \cite[Def. 3.1]{ma},
we define
\begin{equation}
\label{eq-def-Tfilt}
\mathscr{T}\big(E_\bullet,g^{F_\bullet},g^E\big) =
- \sum_{j=1}^r  \mathscr{T}\big(g^{E_{j-1}},g^{E_j},g^{F_j}\big) \;.
\end{equation}
By \eqref{intro-fsum}, \eqref{eq-dtf} and \eqref{eq-def-Tfilt},
we have
\begin{equation}
\label{eq0-dlimtf}
d \mathscr{T}\big(E_\bullet,g^{F_\bullet},g^E\big) = f\big(\n^E,g^E\big) \;.
\end{equation}

Recall that $Q^{S,0}$ was defined in \eqref{intro-Q0}.
We will use the notation in \eqref{intro-deg}.

\begin{thm}
\label{thm-ind}
For
\begin{equation}
E_\bullet': \; E = E_s' \supseteq E_{s-1}' \supseteq \cdots \supseteq E_1' \supseteq E_0' = 0
\end{equation}
another filtration with unitary factors
and $g^{F_j'}$ flat Hermitian metric on $F_j' := E_j'/E_{j-1}'$,
we have
\begin{equation}
\label{eq-thm-ind}
\mathscr{T}^{[>0]}\big(E_\bullet',g^{F_\bullet'},g^E\big) - \mathscr{T}^{[>0]}\big(E_\bullet,g^{F_\bullet},g^E\big) \in Q^{S,0} \;.
\end{equation}
\end{thm}
\begin{proof}
The proof consists of several steps.

\vspace{1.5mm}

\noindent\textbf{Step 1.}
We prove \eqref{eq-thm-ind} under the assumption that $r=s$ and $E_j' = E_j$ for each $j$.

From \eqref{eq-def-Tfilt} and our assumptions,
we get
\begin{align}
\label{eq11-pf-thm-ind}
\begin{split}
& \mathscr{T}\big(E_\bullet',g^{F_\bullet'},g^E\big) - \mathscr{T}\big(E_\bullet,g^{F_\bullet},g^E\big) \\
& = - \sum_{j=1}^r \Big(\mathscr{T}\big(g^{E_{j-1}},g^{E_j},g^{F_j'}\big) - \mathscr{T}\big(g^{E_{j-1}},g^{E_j},g^{F_j}\big)\Big) \;.
\end{split}
\end{align}
We consider the following commutative diagram of flat complex vector bundles over $S$ with exact rows and columns,
\begin{align}
\label{eq1diag-pf-thm-ind}
\begin{split}
\xymatrix{
0 \ar[r] & E_{j-1} \ar[r] & E_j \ar[r] & F_j \ar[r] & 0 \\
0 \ar[r] & E_{j-1} \ar[r] \ar[u]_{\mathrm{Id}} & E_j \ar[r] \ar[u]_{\mathrm{Id}} & F_j \ar[r] \ar[u]_{\mathrm{Id}} & 0 \;,}
\end{split}
\end{align}
where all the maps are canonical embeddings and projections.
We equip $E_{j-1},E_j$ in \eqref{eq1diag-pf-thm-ind} with the metrics $g^{E_{j-1}},g^{E_j}$,
equip the lower $F_j$ in \eqref{eq1diag-pf-thm-ind} with the metric $g^{F_j}$,
and equip the upper $F_j$ in \eqref{eq1diag-pf-thm-ind} with the metric $g^{F_j'}$.
Applying \cite[Theorem A1.4]{bl} to \eqref{eq1diag-pf-thm-ind},
we obtain
\begin{equation}
\label{eq12-pf-thm-ind}
\mathscr{T}\big(g^{F_j},g^{F_j'}\big)
= \mathscr{T}\big(g^{E_{j-1}},g^{E_j},g^{F_j}\big) - \mathscr{T}\big(g^{E_{j-1}},g^{E_j},g^{F_j'}\big) \hspace{2.5mm} \text{modulo } Q^{S,0} \;.
\end{equation}
On the other hand,
since $g^{F_j},g^{F_j'}$ are flat Hermitian metrics,
$\mathscr{T}\big(g^{F_j},g^{F_j'}\big)$ is a locally constant function on $S$.
As a consequence,
we have
\begin{equation}
\label{eq13-pf-thm-ind}
\mathscr{T}^{[>0]}\big(g^{F_j},g^{F_j'}\big) = 0 \;.
\end{equation}
From \eqref{eq11-pf-thm-ind}, \eqref{eq12-pf-thm-ind} and \eqref{eq13-pf-thm-ind},
we obtain \eqref{eq-thm-ind}.

\vspace{1.5mm}

\noindent\textbf{Step 2.}
We prove \eqref{eq-thm-ind} under the assumption that the filtration $E_\bullet'$ is a refinement of $E_\bullet$, i.e.,
there is a map $\sigma: \big\{1,\cdots,r\big\} \rightarrow \big\{1,\cdots,s\big\}$ such that $E_j = E_{\sigma(j)}'$ for each $j$.

We may and we will assume that $s=r+1$ and there exists $a\in\N$ such that
\begin{equation}
\label{eq21-pf-thm-ind}
E_j' = E_j \hspace{2.5mm} \text{for } j=1,\cdots,a-1 \;,\hspace{5mm}
E_j' = E_{j-1} \hspace{2.5mm} \text{for } j=a+1,\cdots,s \;.
\end{equation}
As a consequence,
we have
\begin{equation}
\label{eq22-pf-thm-ind}
F_j' = F_j \hspace{2.5mm} \text{for } j=1,\cdots,a-1 \;,\hspace{5mm}
F_j' = F_{j-1}  \hspace{2.5mm} \text{for } j=a+2,\cdots,s \;.
\end{equation}
By Step 1,
we may further assume that
\begin{equation}
\label{eq23-pf-thm-ind}
g^{F_j'} = g^{F_j} \hspace{2.5mm} \text{for } j=1,\cdots,a-1 \;,\hspace{5mm}
g^{F_j'} = g^{F_{j-1}}  \hspace{2.5mm} \text{for } j=a+2,\cdots,s \;.
\end{equation}
By \eqref{eq-def-Tfilt} and \eqref{eq21-pf-thm-ind}-\eqref{eq23-pf-thm-ind},
we have
\begin{align}
\label{eq24-pf-thm-ind}
\begin{split}
& \mathscr{T}\big(E_\bullet',g^{F_\bullet'},g^E\big) - \mathscr{T}\big(E_\bullet,g^{F_\bullet},g^E\big) \\
& = \mathscr{T}\big(g^{E_{a-1}},g^{E_a},g^{F_a}\big)
- \mathscr{T}\big(g^{E_{a-1}},g^{E_a'},g^{F_a'}\big)
- \mathscr{T}\big(g^{E_a'},g^{E_a},g^{F_{a+1}'}\big) \;.
\end{split}
\end{align}

We consider the following commutative diagram of flat complex vector bundles over $S$ with exact rows and columns,
\begin{align}
\label{eq2diag-pf-thm-ind}
\begin{split}
\xymatrix{
& & 0 & 0 & \\
& & F_{a+1}'  \ar[u] \ar[r]^{\mathrm{Id}} & F_{a+1}' \ar[u] & \\
0 \ar[r] & E_{a-1} \ar[r] & E_a \ar[r]  \ar[u] & F_a \ar[r] \ar[u] & 0 \\
0 \ar[r] & E_{a-1} \ar[r] \ar[u]_{\mathrm{Id}} & E_a' \ar[r] \ar[u] & F_a' \ar[r] \ar[u] & 0 \\
& & 0 \ar[u] & 0 \ar[u] &  \;,}
\end{split}
\end{align}
where all the maps are canonical embeddings and projections.
Applying \cite[Theorem A1.4]{bl} to \eqref{eq2diag-pf-thm-ind},
we obtain
\begin{align}
\label{eq25-pf-thm-ind}
\begin{split}
& \mathscr{T}\big(g^{F_a'},g^{F_a},g^{F_{a+1}'}\big)
- \mathscr{T}\big(g^{E_a'},g^{E_a},g^{F_{a+1}'}\big) \\
& = \mathscr{T}\big(g^{E_{a-1}},g^{E_a'},g^{F_a'}\big)
- \mathscr{T}\big(g^{E_{a-1}},g^{E_a},g^{F_a}\big) \hspace{2.5mm} \text{modulo } Q^{S,0} \;.
\end{split}
\end{align}
On the other hand,
since $g^{F_a'},g^{F_a},g^{F_{a+1}'}$ are flat Hermitian metrics,
$\mathscr{T}\big(g^{F_a'},g^{F_a},g^{F_{a+1}'}\big)$ is a locally constant function on $S$.
As a consequence,
we have
\begin{equation}
\label{eq26-pf-thm-ind}
\mathscr{T}^{[>0]}\big(g^{F_a'},g^{F_a},g^{F_{a+1}'}\big) = 0 \;.
\end{equation}
From \eqref{eq24-pf-thm-ind}, \eqref{eq25-pf-thm-ind} and \eqref{eq26-pf-thm-ind},
we obtain \eqref{eq-thm-ind}.

\vspace{1.5mm}

\noindent\textbf{Step 3.}
We prove \eqref{eq-thm-ind} under the assumption that
$r=s$ and there exists $a\in\N$ such that $E_j = E_j'$ for $j\neq a$.

By Step 2,
we may replace the filtration $E_\bullet$ by
\begin{equation}
E_r \supseteq \cdots \supseteq E_{a+1} \supseteq E_a+E_a' \supseteq E_a \supseteq E_a \cap E_a' \supseteq E_{a-1} \supseteq \cdots \supseteq E_0 = 0 \;,
\end{equation}
and replace the filtration $E_\bullet'$ by
\begin{equation}
E_r \supseteq \cdots \supseteq E_{a+1} \supseteq E_a+E_a' \supseteq E_a' \supseteq E_a \cap E_a' \supseteq E_{a-1} \supseteq \cdots \supseteq E_0 = 0 \;.
\end{equation}
Equivalently,
we may and we will assume that
\begin{equation}
E_{a+1} = E_{a+1}' = E_a + E_a' \;,\hspace{5mm}
E_{a-1} = E_{a-1}' = E_a \cap E_a' \;.
\end{equation}
As a consequence,
we have
\begin{equation}
E_{a+1}/E_{a-1} = F_a \oplus F_a' \;.
\end{equation}
Hence
\begin{equation}
\label{eq3f-pf-thm-ind}
E_r \supseteq \cdots \supseteq E_{a+1} \supseteq E_{a-1} \supseteq \cdots \supseteq E_0 = 0
\end{equation}
is a filtration with unitary factors.
Note that both $E_\bullet$ and $E_\bullet'$ are refinements of the filtration \eqref{eq3f-pf-thm-ind},
from Step 2,
we obtain \eqref{eq-thm-ind}.

\vspace{1.5mm}

\noindent\textbf{Step 4.}
We prove \eqref{eq-thm-ind}.

Similarly to the proof of Jordan-H{\"o}lder theorem,
we can find $E_\bullet^0,\cdots,E_\bullet^m$ filtrations of $E$ by flat subbundles with unitary factors such that
\begin{itemize}
\item[-] for each $i,j$, we have $E^i_j\in\big\{ E_k \cap E_l' \;:\; k=0,\cdots,r,\; l = 0,\cdots,s \big\}$;
\item[-] for $i=0,\cdots,m-1$, there exists $a_i\in\N$ such that $E_j^i = E_j^{i+1}$ for $j\neq a_i$;
\item[-] $E_\bullet^0$ is a refinement of $E_\bullet$, and $E_\bullet^m$ is a refinement of $E_\bullet'$.
\end{itemize}
For each $i,j$,
let $F^i_j = E^i_j/E^i_{j-1}$,
let $g^{F^i_j}$ be a flat Hermitian metric on $F^i_j$.
By Step 2,
we have
\begin{align}
\label{eq41-pf-thm-ind}
\begin{split}
& \mathscr{T}^{[>0]}\big(E_\bullet,g^{F_\bullet},g^E\big) - \mathscr{T}^{[>0]}\big(E_\bullet^0,g^{F_\bullet^0},g^E\big) \in Q^{S,0} \;,\\
& \mathscr{T}^{[>0]}\big(E_\bullet',g^{F_\bullet'},g^E\big) - \mathscr{T}^{[>0]}\big(E_\bullet^m,g^{F_\bullet^m},g^E\big) \in Q^{S,0} \;.
\end{split}
\end{align}
On the other hand,
by Step 3,
we have
\begin{equation}
\label{eq42-pf-thm-ind}
\mathscr{T}^{[>0]}\big(E_\bullet^i,g^{F^i_\bullet},g^E\big) - \mathscr{T}^{[>0]}\big(E_\bullet^{i+1},g^{F_\bullet^{i+1}},g^E\big) \in Q^{S,0}
\hspace{5mm} \text{for } i=0,\cdots,m-1 \;.
\end{equation}
From \eqref{eq41-pf-thm-ind} and \eqref{eq41-pf-thm-ind},
we obtain \eqref{eq-thm-ind}.
This completes the proof.
\end{proof}

\begin{defn}
\label{def-limtf}
We define
\begin{equation}
\label{eq-def-limtf}
\mathscr{T}\big(E,g^E\big) = \mathscr{T}^{[>0]}\big(E_\bullet,g^{F_\bullet},g^E\big) \in Q^S/Q^{S,0} \;.
\end{equation}
By Theorem \ref{thm-ind},
$\mathscr{T}\big(E,g^E\big)$ is well-defined.
\end{defn}

By Definition \ref{def-limtf} and \eqref{eq0-dlimtf},
we have
\begin{equation}
\label{eq-dlimtf}
d \mathscr{T}\big(E,g^E\big) = f\big(\n^E,g^E\big)^{[>1]} \;.
\end{equation}

Similarly to Theorem \ref{thm-tf-4kplus2},
we have the following result.

\begin{thm}
\label{thm-limtf-4kplus2}
If the triple $(E,\n^E,g^E)$
is the complexification of a real triple $(E_\R,\n^{E_\R},g^{E_\R})$,
then
\begin{equation}
\label{eq-thm-limtf-4kplus2}
\mathscr{T}^{[k]}\big(E,g^E\big) = 0 \hspace{2.5mm} \text{for } k \equiv 2 \; (\mathrm{mod} \; 4) \;.
\end{equation}
\end{thm}
\begin{proof}
Let
\begin{equation}
E = \overline{E}_r \supseteq \overline{E}_{r-1} \supseteq \cdots \supseteq \overline{E}_1 \supseteq \overline{E}_0 = 0
\end{equation}
be the conjugated filtration.
Since $E_\bullet$ is with unitary factors,
so is $\overline{E}_\bullet$.
We denote $\overline{F}_j = \overline{E}_j/\overline{E}_{j-1}$.
Let $g^{\overline{F}_j}$ be the Hermitian metric on $\overline{F}_j$ defined by $g^{\overline{F}_j}(v,v) = g^{F_j}(\overline{v},\overline{v})$.
Since $g^{F_j}$ is a flat Hermitian metric,
so is $\overline{F}_j$.
Let $g^{\overline{E}_j}$ be the restricted metric of $g^E$ on $\overline{E}_j$.
Similarly to Theorem \ref{thm-tf-4kplus2},
the proof of \cite[Thm. 1.8(iv)]{bl} yields
\begin{align}
\label{eq1-pf-thm-limtf-4kplus2}
\begin{split}
& \mathscr{T}^{[k]}\big(g^{\overline{E}_{j-1}},g^{\overline{E}_j},g^{\overline{F}_j}\big) \\
& = \overline{\mathscr{T}^{[k]}\big(g^{E_{j-1}},g^{E_j},g^{F_j}\big)}
= (-1)^{k(k-1)/2} \mathscr{T}^{[k]}\big(g^{E_{j-1}},g^{E_j},g^{F_j}\big) \;.
\end{split}
\end{align}
By Definition \ref{def-limtf}, \eqref{eq-def-Tfilt} and \eqref{eq1-pf-thm-limtf-4kplus2},
for $k>0$,
we have
\begin{align}
\label{eq2-pf-thm-limtf-4kplus2}
\begin{split}
& \mathscr{T}^{[k]}\big(E,g^E\big)
= - \sum_{j=1}^r  \mathscr{T}^{[k]}\big(g^{\overline{E}_{j-1}},g^{\overline{E}_j},g^{\overline{F}_j}\big) \\
& = - (-1)^{k(k-1)/2} \sum_{j=1}^r   \mathscr{T}^{[k]}\big(g^{E_{j-1}},g^{E_j},g^{F_j}\big)
= (-1)^{k(k-1)/2} \mathscr{T}^{[k]}\big(E,g^E\big) \;.
\end{split}
\end{align}
From \eqref{eq2-pf-thm-limtf-4kplus2},
we obtain \eqref{eq-thm-limtf-4kplus2}.
This completes the proof.
\end{proof}

Now we consider an exact sequence of flat complex vector bundles over $S$,
\begin{equation}
\label{eq-Epartial}
(E^\bullet,\partial): \; 0 \rightarrow E^0 \rightarrow E^1 \rightarrow \cdots \rightarrow E^n \rightarrow 0 \;.
\end{equation}
We assume that each $E^k$ is filtered by flat subbundles with unitary factors.

Let $g^{E^\bullet} = \bigoplus_k g^{E^k}$ be a Hermitian metric on $E^\bullet$.
We define
\begin{equation}
\label{eq2-def-limtf}
\mathscr{T}\big(E^\bullet,g^{E^\bullet}\big) = \sum_{k=0}^n (-1)^k \mathscr{T}\big(E^k,g^{E^k}\big) \in Q^S/Q^{S,0} \;.
\end{equation}
By \eqref{intro-fsum}, \eqref{eq-dlimtf} and \eqref{eq2-def-limtf},
we have
\begin{equation}
\label{eq2-dlimtf}
d \mathscr{T}\big(E^\bullet,g^{E^\bullet}\big) = f\big(\n^{E^\bullet},g^{E^\bullet}\big)^{[>1]} \;.
\end{equation}
Let $\mathscr{T}\big(E^\bullet,\partial,g^{E^\bullet}\big)$ be the torsion form of $(E^\bullet,\partial)$
equipped with the metric $g^{E^\bullet}$.
We have
\begin{equation}
\label{eq3-dlimtf}
d \mathscr{T}\big(E^\bullet,\partial,g^{E^\bullet}\big) = f\big(\n^{E^\bullet},g^{E^\bullet}\big) \;.
\end{equation}
From \eqref{eq2-dlimtf} and \eqref{eq3-dlimtf},
we see that $\mathscr{T}\big(E^\bullet,g^{E^\bullet}\big) - \mathscr{T}^{[>0]}\big(E^\bullet,\partial,g^{E^\bullet}\big)$ is closed.

\begin{thm}
\label{thm-limtf}
The following identity holds in $Q^S/Q^{S,0}$,
\begin{equation}
\label{eq-thm-limtf}
\mathscr{T}\big(E^\bullet,g^{E^\bullet}\big) =
\mathscr{T}^{[>0]}\big(E^\bullet,\partial,g^{E^\bullet}\big) \;.
\end{equation}
\end{thm}
\begin{proof}
The proof consists of two steps.

\vspace{1.5mm}

\noindent\textbf{Step 1.}
We prove \eqref{eq-thm-limtf} under the assumption that $n=2$, i.e.,
\eqref{eq-Epartial} is given by
\begin{equation}
0 \rightarrow E^0 \xrightarrow{\alpha} E^1 \xrightarrow{\beta} E^2 \rightarrow 0 \;.
\end{equation}

We can find filtrations with unitary factors
\begin{align}
\begin{split}
& E^0 = E^0_r \supseteq \cdots \supseteq E^0_0 = 0 \;,\hspace{5mm}
E^1 = E^1_{r+s} \supseteq \cdots \supseteq E^1_0 = 0 \;,\\
& E^2 = E^2_{r+s} \supseteq \cdots \supseteq E^2_r = 0
\end{split}
\end{align}
such that
\begin{equation}
E^1_j = \alpha(E^0_j) \hspace{2.5mm} \text{for } j=1,\cdots,r \;,\hspace{5mm}
E^2_j = \beta(E^1_j) \hspace{2.5mm} \text{for } j=r+1,\cdots,r+s \;.
\end{equation}
For each $i,j$,
we denote $F^i_j = E^i_j/E^i_{j-1}$.
Then we have
\begin{equation}
F^0_j = F^1_j \hspace{2.5mm} \text{for } j=1,\cdots,r \;,\hspace{5mm}
F^2_j = F^1_j \hspace{2.5mm} \text{for } j=r+1,\cdots,r+s \;.
\end{equation}
For ease of notations,
we denote $F_j = F^1_j$.
For each $j$,
let $g^{F_j}$ be a flat Hermitian metric on $F_j$.
For each $i,j$,
we denote by $g^{E^i_j}$ the restricted metric of $g^{E^i}$ on $E^i_j$.
By Definition \ref{def-limtf}, \eqref{eq-def-Tfilt} and \eqref{eq2-def-limtf},
we have
\begin{align}
\label{eq1a-pf-thm-limtf}
\begin{split}
\mathscr{T}\big(E^\bullet,g^{E^\bullet}\big)
& = \sum_{j=1}^r  \Big( \mathscr{T}^{[>0]}\big(g^{E_{j-1}^1},g^{E_j^1},g^{F_j}\big) - \mathscr{T}^{[>0]}\big(g^{E_{j-1}^0},g^{E_j^0},g^{F_j}\big) \Big) \\
& \hspace{5mm} + \sum_{j=r+1}^{r+s}  \Big( \mathscr{T}^{[>0]}\big(g^{E_{j-1}^1},g^{E_j^1},g^{F_j}\big) - \mathscr{T}^{[>0]}\big(g^{E_{j-1}^2},g^{E_j^2},g^{F_j}\big) \Big) \;.
\end{split}
\end{align}

For $j=1,\cdots,r$,
applying \cite[Theorem A1.4]{bl} to the commutative diagram with exact rows and columns
\begin{align}
\begin{split}
\xymatrix{
0 \ar[r] & E_{j-1}^1 \ar[r] & E_j^1 \ar[r]  & F_j \ar[r] & 0 \\
0 \ar[r] & E_{j-1}^0 \ar[r] \ar[u] & E_j^0 \ar[r] \ar[u] & F_j \ar[r] \ar[u]_{\mathrm{Id}} & 0 \;,}
\end{split}
\end{align}
we obtain
\begin{align}
\begin{split}
& \mathscr{T}\big(g^{E_{j-1}^1},g^{E_j^1},g^{F_j}\big) - \mathscr{T}\big(g^{E_{j-1}^0},g^{E_j^0},g^{F_j}\big) \\
& = \mathscr{T}\big(g^{E_j^0},g^{E_j^1}\big) - \mathscr{T}\big(g^{E_{j-1}^0},g^{E_{j-1}^1}\big)
\hspace{2.5mm} \text{modulo } Q^{S,0} \;.
\end{split}
\end{align}
As a consequence,
we have
\begin{equation}
\label{eq1b-pf-thm-limtf}
\sum_{j=1}^r  \Big( \mathscr{T}\big(g^{E_{j-1}^1},g^{E_j^1},g^{F_j}\big) - \mathscr{T}\big(g^{E_{j-1}^0},g^{E_j^0},g^{F_j}\big) \Big)
= \mathscr{T}\big(g^{E^0},g^{E_r^1}\big)
\hspace{2.5mm} \text{modulo } Q^{S,0} \;.
\end{equation}

For $j=r+1,\cdots,r+s$,
applying \cite[Theorem A1.4]{bl} to the commutative diagram with exact rows and columns
\begin{align}
\begin{split}
\xymatrix{
& 0 & 0 &  & \\
0 \ar[r] & E_{j-1}^2 \ar[r] \ar[u] & E_j^2 \ar[r] \ar[u] & F_j \ar[r] & 0 \\
0 \ar[r] & E_{j-1}^1 \ar[r] \ar[u] & E_j^1 \ar[r] \ar[u] & F_j \ar[r] \ar[u]_{\mathrm{Id}} & 0 \\
& E^0 \ar[r]^{\mathrm{Id}} \ar[u] & E^0 \ar[u] & &  \\
& 0 \ar[u] & 0 \ar[u] & & \;,}
\end{split}
\end{align}
we obtain
\begin{align}
\begin{split}
& \mathscr{T}\big(g^{E_{j-1}^1},g^{E_j^1},g^{F_j}\big) - \mathscr{T}\big(g^{E_{j-1}^2},g^{E_j^2},g^{F_j}\big) \\
& = \mathscr{T}\big(g^{E^0},g^{E_j^1},g^{E_j^2}\big) - \mathscr{T}\big(g^{E^0},g^{E_{j-1}^1},g^{E_{j-1}^2}\big)
\hspace{2.5mm} \text{modulo } Q^{S,0} \;.
\end{split}
\end{align}
As a consequence,
we have
\begin{align}
\label{eq1c-pf-thm-limtf}
\begin{split}
& \sum_{j=r+1}^{r+s}  \Big( \mathscr{T}\big(g^{E_{j-1}^1},g^{E_j^1},g^{F_j}\big) - \mathscr{T}\big(g^{E_{j-1}^2},g^{E_j^2},g^{F_j}\big) \Big) \\
& = \mathscr{T}\big(g^{E^0},g^{E^1},g^{E^2}\big) - \mathscr{T}\big(g^{E^0},g^{E_r^1}\big)
\hspace{2.5mm} \text{modulo } Q^{S,0} \;.
\end{split}
\end{align}
From \eqref{eq1a-pf-thm-limtf}, \eqref{eq1b-pf-thm-limtf}, \eqref{eq1c-pf-thm-limtf},
we obtain
\begin{equation}
\mathscr{T}\big(E^\bullet,g^{E^\bullet}\big) = \mathscr{T}^{[>0]}\big(g^{E^0},g^{E^1},g^{E^2}\big)
\hspace{2.5mm} \text{modulo } Q^{S,0} \;,
\end{equation}
which is equivalent to \eqref{eq-thm-limtf} with $n=2$.

\vspace{1.5mm}

\noindent\textbf{Step 2.}
We prove \eqref{eq-thm-limtf} by induction on $n$.

The cases $n = 1,2$ are proved in Step 1.
In the sequel,
we assume that $n \geqslant 3$.
Let $K^{n-1}\subseteq E^{n-1}$ be the kernel of $E^{n-1}\rightarrow E^n$,
which is also the image of $E^{n-2}\rightarrow E^{n-1}$.
We have an exact sequence
\begin{equation}
(\widetilde{E}^\bullet,\partial): \; 0 \rightarrow E^0 \rightarrow \cdots \rightarrow E^{n-2} \rightarrow K^{n-1} \rightarrow 0 \;.
\end{equation}
Let $g^{K^{n-1}}$ be the restricted metric of $g^{E^{n-1}}$ on $K^{n-1}$.
Let $g^{\widetilde{E}^\bullet}$ be the metric on $\widetilde{E}^\bullet$ defined by $\big(g^{E^i}\big)_{i=0,\cdots,n-2}$ and $g^{K^{n-1}}$.
Let $\mathscr{T}\big(\widetilde{E}^\bullet,\partial,g^{\widetilde{E}^\bullet}\big)$
be the torsion form of $(\widetilde{E}^\bullet,\partial)$ equipped with the metric $g^{\widetilde{E}^\bullet}$.

Applying \cite[Theorem A1.4]{bl} to the commutative diagram with exact rows and columns
\begin{align}
\begin{split}
\xymatrix{
& & & & 0 & & \\
& & & &  E^n \ar[r]^{\mathrm{Id}} \ar[u] & E^n & \\
0 \ar[r] & E^0 \ar[r] & \cdots \ar[r] & E^{n-2} \ar[r] &  E^{n-1} \ar[r] \ar[u] & E^n \ar[r] \ar[u]_{\mathrm{Id}} & 0 \\
0 \ar[r] & E^0 \ar[r] \ar[u]_{\mathrm{Id}} & \cdots \ar[r] & E^{n-2} \ar[u]_{\mathrm{Id}} \ar[r] &  K^{n-1} \ar[r] \ar[u] & 0 &  \\
& & & & 0 \ar[u] & \;, &  }
\end{split}
\end{align}
we obtain
\begin{align}
\label{eq2a-pf-thm-limtf}
\begin{split}
& \mathscr{T}\big(E^\bullet,\partial,g^{E^\bullet}\big) - \mathscr{T}\big(\widetilde{E}^\bullet,\partial,g^{\widetilde{E}^\bullet}\big) \\
& = (-1)^n \mathscr{T}\big(g^{K^{n-1}},g^{E^{n-1}},g^{E^n}\big)
\hspace{2.5mm} \text{modulo } Q^{S,0} \;.
\end{split}
\end{align}
On the other hand,
by Step 1 and \eqref{eq2-def-limtf},
we have
\begin{align}
\label{eq2b-pf-thm-limtf}
\begin{split}
& \mathscr{T}\big(E^\bullet,g^{E^\bullet}\big) - \mathscr{T}\big(\widetilde{E}^\bullet,g^{\widetilde{E}^\bullet}\big) \\
& = (-1)^n \Big( \mathscr{T}\big(K^{n-1},g^{K^{n-1}}\big) - \mathscr{T}\big(E^{n-1},g^{E^{n-1}}\big) + \mathscr{T}\big(E^n,g^{E^n}\big) \Big) \\
& = (-1)^n \mathscr{T}^{[>0]}\big(g^{K^{n-1}},g^{E^{n-1}},g^{E^n}\big)
\hspace{2.5mm} \text{modulo } Q^{S,0} \;.
\end{split}
\end{align}
From \eqref{eq2a-pf-thm-limtf} and \eqref{eq2b-pf-thm-limtf},
we obtain
\begin{align}
\begin{split}
& \mathscr{T}^{[>0]}\big(E^\bullet,\partial,g^{E^\bullet}\big) - \mathscr{T}^{[>0]}\big(\widetilde{E}^\bullet,\partial,g^{\widetilde{E}^\bullet}\big) \\
& = \mathscr{T}\big(E^\bullet,g^{E^\bullet}\big) - \mathscr{T}\big(\widetilde{E}^\bullet,g^{\widetilde{E}^\bullet}\big)
\hspace{2.5mm} \text{modulo } Q^{S,0} \;.
\end{split}
\end{align}
This completes the proof by induction.
\end{proof}

\subsection{Analytic torsion form}
This subsection follows \cite[\textsection III]{bl}.

Let
\begin{equation}
\pi: M \rightarrow S
\end{equation}
be a smooth fibration with compact fiber $Z$.
Here $Z$ may have boundaries.
Let $T^HM \subseteq TM$ be a complement of $TZ$,
which induces the following identification,
\begin{equation}
\label{eq-id-wedge}
\Lambda^\bullet(T^*M)
= \Lambda^\bullet(T^{H,*}M) \otimes \Lambda^\bullet(T^*Z)
\simeq \pi^*\big(\Lambda^\bullet(T^*S)\big) \otimes \Lambda^\bullet(T^*Z) \;.
\end{equation}

Let
\begin{equation}
(F,\n^F)
\end{equation}
be a flat complex vector bundle over $M$.
Set $\mathscr{F}^\bullet = \Omega^\bullet(Z,F)$,
which is a graded complex vector bundle of infinite dimension over $S$.
From \eqref{eq-id-wedge},
we get a formal identity
\begin{equation}
\Omega^\bullet(M,F) = \Omega^\bullet(S,\mathscr{F}^\bullet) \;.
\end{equation}

For $U\in TS$,
let $U^H \in T^HM$ be such that $\pi_*U^H = U$.
For $U\in\smooth(S,TS)$,
let $L_{U^H}$ be the Lie differentiation operator acting on $\Omega^\bullet(M,F)$.
For $U\in\smooth(S,TS)$ and $s\in\Omega^\bullet(S,\mathscr{F}^\bullet) = \Omega^\bullet(M,F)$,
we define
\begin{equation}
\label{eq-def-n-F}
\n^{\mathscr{F}^\bullet}_U s = L_{U^H} s \;.
\end{equation}
Then $\n^{\mathscr{F}^\bullet}$ is a connection on $\mathscr{F}^\bullet$ preserving the grading.

Let $P^{TZ}: TM \rightarrow TZ$ be the projection with respect to the decomposition $TM = T^HM \oplus TZ$.
For $U,V\in\smooth(S,TS)$,
set
\begin{equation}
\label{eq-def-mathcalT}
\mathcal{T}(U,V) = - P^{TZ}[U^H,V^H] \in \smooth(M,TZ) \;.
\end{equation}
Then
\begin{equation}
\mathcal{T}\in\smooth\big(M,\pi^*\big(\Lambda^2(T^*S)\big)\otimes TZ\big) \;.
\end{equation}
Let $i_\mathcal{T}\in\smooth\big(M,\pi^*\big(\Lambda^2(T^*S)\big)\otimes\mathrm{End}\big(\Lambda^\bullet(T^*Z)\big)\big)$
be the interior multiplication by $\mathcal{T}$.

The flat connection $\n^F$ (resp. $\n^F\big|_Z$) naturally extends to an exterior differentiation operator
on $\Omega^\bullet(M,F)$ (resp. $\Omega^\bullet(Z,F) = \mathscr{F}^\bullet$),
which we denote by $d^M$ (resp. $d^Z$).
In the sense of \cite[\textsection II(a)]{bl},
the operator $d^M$ is a superconnection of total degree $1$ on $\mathscr{F}^\bullet$.
By \cite[Prop. 3.4]{bl},
we have
\begin{equation}
\label{eq-dM}
d^M = d^Z + \n^{\mathscr{F}^\bullet} + i_\mathcal{T} \;.
\end{equation}

Let $g^{TZ}$ be a Riemannian metric on $TZ$.
Let $g^F$ be a Hermitian metric on $F$.

We denote $Y = \partial Z$.
If $Y \neq \emptyset$,
we assume that $T^HM$, $g^{TZ}$ and $g^F$ are product on a tubular neighborhood of $Y$ (see \cite[(2.30)-(2.35)]{imrn-zhu}).

Let $g^{\mathscr{F}^\bullet}$ be the $L^2$-metric on $\mathscr{F}^\bullet$ associated with $g^{TZ}, g^F$.
Let $d^{M,*}, d^{Z,*},\n^{\mathscr{F}^\bullet,*}$
be the formal adjoints of  $d^M,d^Z,\n^{\mathscr{F}^\bullet}$ with respect to $g^{\mathscr{F}^\bullet}$
in the sense of \cite[Def. 1.6]{bl}.
By \cite[Prop. 3.7]{bl},
we have
\begin{equation}
d^{M,*} = d^{Z,*} + \n^\mathscr{F,*} - \mathcal{T}^*\!\wedge \;,
\end{equation}
where $\mathcal{T}^*\in\smooth\big(M,\pi^*\big(\Lambda^2(T^*S)\big)\otimes T^*Z\big)$ is the dual of $\mathcal{T}$ with respect to $g^{TZ}$.

Let $N^{TZ}$ be the number operator on $\Lambda^\bullet(T^*Z)$,
i.e., $N^{TZ}\big|_{\Lambda^k(T^*Z)} = k \Id$.
Then $N^{TZ}$ acts on $\mathscr{F}^\bullet$ in the obvious way.
For $t>0$,
let $d^{M,*}_t$ be the formal adjoints of  $d^M$ with respect to the metric $t^{N^{TZ}}g^{\mathscr{F}^\bullet}$.
We have
\begin{equation}
\label{eq-dMstar-t}
d^{M,*}_t = t d^{Z,*} + \n^\mathscr{F,*} - \frac{1}{t} \mathcal{T}^*\!\wedge \;.
\end{equation}
Set
\begin{equation}
\label{eq-D-t}
D_t = \frac{1}{2}\big( d^{M,*}_t - d^M \big) \;.
\end{equation}
By \eqref{eq-dM}, \eqref{eq-dMstar-t} and \eqref{eq-D-t},
we have
\begin{equation}
\label{eq-Dt2}
D_t^2 = - \frac{t}{4} \big(d^{Z,*}d^Z+d^Zd^{Z,*}\big) + \text{positive degree terms} \;.
\end{equation}
Here $d^{Z,*}d^Z+d^Zd^{Z,*}$ is the fiberwise Hodge Laplacian.
If $Y = \partial Z \neq \emptyset$,
we put absolute boundary condition (see \cite[(2.52)]{imrn-zhu}) on $Y$,
then $d^{Z,*}d^Z+d^Zd^{Z,*}$ is self-adjoint.

Let $\mathrm{End}_\mathrm{tr}(\mathscr{F}^\bullet) \subseteq \mathrm{End}(\mathscr{F}^\bullet)$
be the subbundle of trace class operators.
Recall that $f'(z) = (1+2z^2)e^{z^2}$.
By \eqref{eq-Dt2},
we have
\begin{equation}
f'(D_t) \in \Omega^\bullet\big(S,\mathrm{End}_\mathrm{tr}(\mathscr{F}^\bullet)\big) \;.
\end{equation}
Let $\tr: \mathrm{End}_\mathrm{tr}(\mathscr{F}^\bullet) \rightarrow \C$ be the trace map,
which extends to
\begin{equation}
\tr: \Omega^\bullet\big(S,\mathrm{End}_\mathrm{tr}(\mathscr{F}^\bullet)\big) \rightarrow \Omega^\bullet(S) \;.
\end{equation}

Let $H^\bullet(Z,F)$ be the fiberwise cohomology of $Z$ with coefficients in $F$.
Then $H^\bullet(Z,F)$ is a graded complex vector bundle over $S$.
We denote
\begin{equation}
\label{eq-def-chiprim}
\chi'(Z,F) = \sum_{k=0}^{\dim Z} (-1)^kk\, \mathrm{rk}\big(H^k(Z,F)\big) \;.
\end{equation}

Recall that $\varphi$ was defined in \eqref{eq-def-varphi}.
Recall that $Q^S$ was defined in \eqref{intro-Q}.

Bismut and Lott \cite[Def. 3.22]{bl} gave the following definition for $Z$ closed.
Zhu \cite[Def. 2.18]{imrn-zhu} extended the definition to the case $Y = \partial Z \neq \emptyset$.

\begin{defn}
\label{def-atf}
The analytic torsion form associated with $(T^HM,g^{TZ},g^F)$ is defined as
\begin{align}
\label{eq-def-atf}
\begin{split}
& \mathscr{T}\big(T^HM,g^{TZ},g^F\big) \\
& = - \int_0^{+\infty}
\bigg\{ \varphi \tr\Big[(-1)^{N^{TZ}}\frac{N^{TZ}}{2}f'(D_t)\Big]
- \frac{\chi'(Z,F)}{2} \\
& \hspace{30mm} - \Big(\frac{\dim Z \mathrm{rk}(F)\chi(Z)}{4} - \frac{\chi'(Z,F)}{2}\Big)f'\Big(\frac{i\sqrt{t}}{2}\Big) \bigg\}
\frac{dt}{t} \in Q^S \;.
\end{split}
\end{align}
By \cite[Thm 3.21]{bl} \cite[Thm 2.17]{imrn-zhu},
the integrand in \eqref{eq-def-atf} is integrable.
\end{defn}

Let $\n^{TZ}$ be the Bismut connection \cite[Def. 1.6]{b} on $TZ$ associated with $T^HM$ and $g^{TZ}$.
Let $\n^{TY}$ be the Bismut connection on $TY$ associated with $T^HM$ and $g^{TZ}\big|_{TY}$.
Let $\n^{H^\bullet(Z,F)}$ be the canonical flat connection on $H^\bullet(Z,F)$ (see \cite[Def. 2.4]{bl}).
Let $g^{H^\bullet(Z,F)}$ be the $L^2$-metric on $H^\bullet(Z,F)$ associated with $g^{TZ},g^F$.
By \cite[Thm 3.23]{bl} \cite[Thm 2.19]{imrn-zhu},
we have
\begin{align}
\label{eq-dT}
\begin{split}
& d\mathscr{T}\big(T^HM,g^{TZ},g^F\big) \\
& = \int_Z e\big(TZ,\n^{TZ}\big)f\big(\n^F,g^F\big)
+ \frac{1}{2} \int_Y e\big(TY,\n^{TY}\big)f\big(\n^F,g^F\big) \\
& \hspace{65mm} - f\big(\n^{H^\bullet(Z,F)},g^{H^\bullet(Z,F)}\big) \;.
\end{split}
\end{align}

The following theorem is an infinite dimensional analogue of Theorem \ref{thm-tf-4kplus2}.

We will use the notation in \eqref{intro-deg}.

\begin{thm}
\label{thm-atf-4kplus2}
If the triple $(F,\n^F,g^F)$ is the complexification of a triple $(F_\R,\n^{F_\R},g^{F_\R})$,
then
\begin{equation}
\label{eq-thm-atf-4kplus2}
\mathscr{T}^{[k]}\big(T^HM,g^{TZ},g^F\big) = 0 \hspace{2.5mm} \text{for } k \equiv 2 \; (\mathrm{mod} \; 4) \;.
\end{equation}
\end{thm}
\begin{proof}
Similarly to Theorem \ref{thm-tf-4kplus2},
we apply the argument in \cite[Thm. 1.8(iv)]{bl} with $E^\bullet$ replaced by $\mathscr{F}^\bullet$ and $X_t$ replaced by $D_t$.
Though $\mathscr{F}^\bullet$ is infinite dimensional,
the argument still works.
\end{proof}

\subsection{Igusa's and Ohrt's axioms}
\label{subsect-axiom}

This subsection follows \cite[\textsection 3]{ig2} and \cite[\textsection 2.2]{ohrt}.

We consider an invariant $\tau$
assigning a cohomology class
\begin{equation}
\tau(M/S,F)\in H^\mathrm{even}(S)
\end{equation}
to any triple $\big(\pi: M \rightarrow S, F, \n^F\big)$ satisfying
\begin{itemize}
\item[-] $\pi: M\rightarrow S$ is a smooth fibration with closed fiber;
\item[-] $(F, \n^F\big)$ is a unitarily flat complex vector bundle over $M$;
\item[-] $H^\bullet(Z,F)$ is filtered by flat subbundles with unitary factors.
\end{itemize}
We will state several axioms.
If $\tau$ satisfies all the axioms,
we call $\tau$ a higher torsion invariant.

For ease of notations,
a triple under consideration will be denoted by $F \rightarrow M \rightarrow S$.

\vspace{2.5mm}

Let $F \rightarrow M \rightarrow S$ be a triple under consideration.
Let $S'$ be a smooth manifold.
Let $\varphi: S' \rightarrow S$ be a smooth map.
Let $F' \rightarrow M' \rightarrow S'$ be the pull-back of $F \rightarrow M \rightarrow S$.

\noindent\textbf{Axiom 1} (naturality)
We have
\begin{equation}
\tau(M'/S',F') = \varphi^* \tau(M/S,F) \in  H^\mathrm{even}(S') \;.
\end{equation}

\vspace{2.5mm}

Let $F_j \rightarrow M_j \rightarrow S$ with $j=1,2$ be two triples with boundaries.
More precisely,
$M_1,M_2$ are allowed to have boundaries,
and the other assumptions still hold.
We assume that there is an diffeomorphism $\varphi: \partial M_1 \rightarrow \partial M_2$.
We further assume that there is an isomorphism between flat vector bundles
\begin{equation}
\phi: F\big|_{\partial M_1} \rightarrow \varphi^*\big(F\big|_{\partial M_2}\big) \;.
\end{equation}
We can glue $F_1$ and $F_2$ together to a unitarily flat complex vector bundle $F : = F_1 \cup_\phi F_2$ over $M := M_1 \cup_\varphi M_2$.
Similarly,
for $j=1,2$,
we construct
\begin{equation}
DM_j = M_j \cup_\mathrm{Id} M_j \;,\hspace{5mm} DF_j = F_j \cup_\mathrm{Id} F_j \;.
\end{equation}
Then $DF_j$ is a unitarily flat complex vector bundle over $DM_j$.

\noindent\textbf{Axiom 2} (additivity)
We have
\begin{equation}
\tau(M/S,F) = \frac{1}{2} \tau(DM_1/S,DF_1) + \frac{1}{2} \tau(DM_2/S,DF_2) \;.
\end{equation}

\vspace{2.5mm}

Let $F \rightarrow M \rightarrow S$ be a triple under consideration.
Let $\xi$ be a real vector bundle of rank $n+1$ over $M$.
Let $q: \mathbb{S}^n(\xi) \rightarrow M$ be the $\mathbb{S}^n$-bundle associated.
More precisely,
for any norm $|\cdot|$ on $\xi$,
the manifold $\mathbb{S}^n(\xi)$ is diffeomorphic to $\big\{ v\in\xi \;:\; |v| = 1 \big\}$.

\noindent\textbf{Axiom 3} (transfer)
We have
\begin{equation}
\tau\big(\mathbb{S}^n(\xi)/S,q^*F\big) =
\chi(\mathbb{S}^n) \tau(M/S,F) + \int_Z e(TZ) \tau\big(\mathbb{S}^n(\xi)/M,q^*F\big) \;,
\end{equation}
where $\chi(\mathbb{S}^n)$ is the Euler characteristic of $\mathbb{S}^n$,
$e(TZ)$ is the Euler class of $TZ$,
and $q^*F$ is the pull-back of $F$.

\vspace{2.5mm}

Let $\mathbb{1} \rightarrow M \rightarrow S$ be a triple under consideration,
where $\mathbb{1}$ is the trivial flat line bundle over $M$.
We will use the notation in \eqref{intro-deg}.

\noindent\textbf{Axiom 4} (triviality)
For $k\in\N$,
we have
\begin{equation}
\tau^{[4k+2]}(M/S,\mathbb{1}) = 0 \in H^{4k+2}(S) \;.
\end{equation}

\vspace{2.5mm}

Let $F_j \rightarrow M \rightarrow S$ with $j=1,\cdots,m$ be a family of triples under consideration.

\noindent\textbf{Axiom 5} (additivity of coefficients)
We have
\begin{equation}
\tau\Big(M/S,\bigoplus_{j=1}^m F_j\Big) = \sum_{j=1}^m \tau(M/S,F_j) \;.
\end{equation}

\vspace{2.5mm}

Let $F \rightarrow M \xrightarrow{\pi} S$ be a triple under consideration.
Assume that there is a finite covering $p: M \rightarrow M_\sim$ and a fibration $\pi_\sim: M_\sim \rightarrow S$ such that $\pi = \pi_\sim \circ p$.

\noindent\textbf{Axiom 6} (induction)
We have
\begin{equation}
\tau(M/S,F) = \tau(M_\sim/S,p_*F) \;,
\end{equation}
where $p_*F$ is the direct image of $F$.

\vspace{2.5mm}

We consider the universal complex line bundle $L$ over $\CP^\infty$.
Let $\alpha\in\C$ be a root of unity.
Let $n$ be a positive integer such that $\alpha^n=1$.
Let $L^{\otimes n}$ be the $n$-th tensor product of $L$.
Let $\mathbb{S}^1(L^{\otimes n})$ be the $\mathbb{S}^1$-bundle associated with $L^{\otimes n}$.
Let $F_\alpha$ be a flat complex line bundle over $\mathbb{S}^1(L^{\otimes n})$ such that
\begin{itemize}
\item[-] the pull-back of $F_\alpha$ to $\mathbb{S}^1(L)$ is a trivial flat complex line bundle;
\item[-] the holonomy of $F_\alpha$ along the fiber of $\mathbb{S}^1(L^{\otimes n}) \rightarrow \CP^\infty$ equals $\alpha$.
\end{itemize}

Let $\tau_k\big(\mathbb{S}^1(L^{\otimes n})/\CP^\infty,F_\alpha\big)$ be
the component of $\tau\big(\mathbb{S}^1(L^{\otimes n})/\CP^\infty,F_\alpha\big)$ of degree $2k$.

We identify $\Q/\Z$ with the roots of unity in $\C$.

\noindent\textbf{Axiom 7} (continuity)
For $k\in\N$,
the map
\begin{align}
\label{eq-a7}
\begin{split}
\Q/\Z & \rightarrow H^{2k}(\CP^\infty) = \R \\
\alpha & \mapsto \frac{1}{n^{k}} \tau_k\big(\mathbb{S}^1(L^{\otimes n})/\CP^\infty,F_\alpha\big)
\end{split}
\end{align}
is continuous.
Here,
to show that the map \eqref{eq-a7} is well-defined,
we need Axiom 1 and the following fact:
let $\varphi: \CP^\infty \rightarrow \CP^\infty$ be such that $\varphi^*L \simeq L^{\otimes n}$,
then the pull-back map $\varphi^*: H^{2k}(\CP^\infty) \rightarrow H^{2k}(\CP^\infty)$ equals $n^k\mathrm{Id}$.

\section{Analytic torsion class}

\subsection{Construction}
\label{subsect-con}

Let
\begin{equation}
\pi: M \rightarrow S
\end{equation}
be a smooth fibration with compact fiber.
For $s\in S$,
we denote $Z_s = \pi^{-1}(s)$.
We will omit the index $s$ when we refer to the generic fiber.
Here $Z$ may have boundaries.
Let
\begin{equation}
(F,\n^F)
\end{equation}
be a unitarily flat complex vector bundle over $M$.
Let $H^\bullet(Z,F)$ be the fiberwise cohomology with coefficients in $F$.
Let $\n^{H^\bullet(Z,F)}$ be the canonical flat connection on $H^\bullet(Z,F)$.
We assume that each $H^k(Z,F)$ is filtered by flat subbundles with unitary factors.

Let $T^HM \subseteq TM$ be a complement of $TZ$.
Let $g^{TZ}$ be a Riemannian metric on $TZ$.
Let $g^F$ be a flat Hermitian metric on $F$.
Let $\mathscr{T}\big(T^HM,g^{TZ},g^F\big)$ be the Bismut-Lott analytic torsion form,
which we view as an element in $Q^S/Q^{S,0}$.
Let $g^{H^\bullet(Z,F)}$ be the $L^2$-metric on $H^\bullet(Z,F)$ associated with $g^{TZ},g^F$.
From \eqref{eq-dT} and the assumptions above,
we get
\begin{equation}
\label{eq-datf}
d \mathscr{T}\big(T^HM,g^{TZ},g^F\big) = - f\big(\n^{H^\bullet(Z,F)},g^{H^\bullet(Z,F)}\big) \;.
\end{equation}

Let $\mathscr{T}\big(\n^{H^\bullet(Z,F)},g^{H^\bullet(Z,F)}\big)$ be as in \eqref{eq2-def-limtf}.
By \eqref{eq2-dlimtf},
we have
\begin{equation}
d \mathscr{T}\big(\n^{H^\bullet(Z,F)},g^{H^\bullet(Z,F)}\big) = f\big(\n^{H^\bullet(Z,F)},g^{H^\bullet(Z,F)}\big)^{[>1]} \;.
\end{equation}
We define
\begin{equation}
\label{eq-def-tcl}
\mathscr{T}_\mathrm{cl}\big(T^HM,g^{TZ},g^F\big) =
\mathscr{T}^{[>0]}\big(T^HM,g^{TZ},g^F\big) + \mathscr{T}\big(\n^{H^\bullet(Z,F)},g^{H^\bullet(Z,F)}\big) \;.
\end{equation}
From \eqref{eq-datf}-\eqref{eq-def-tcl},
we get
\begin{equation}
\label{eq-def-dtcl}
d \mathscr{T}_\mathrm{cl}\big(T^HM,g^{TZ},g^F\big) = 0 \;.
\end{equation}

\begin{defn}
\label{def-tclass}
The analytic torsion class of $(\pi: M \rightarrow S, F)$ is defined as
\begin{equation}
\tau^\mathrm{BL}(M/S,F) = \Big[ \mathscr{T}_\mathrm{cl}\big(T^HM,g^{TZ},g^F\big) \Big] \in H^{\mathrm{even}\geqslant 2}(S) \;.
\end{equation}
A standard argument using the functoriality and the closedness of $\mathscr{T}_\mathrm{cl}\big(T^HM,g^{TZ},g^F\big)$
shows that $\tau^\mathrm{BL}(M/S,F)$ is independent of $T^HM,g^{TZ},g^F$.
\end{defn}

If $H^\bullet(Z,F)$ is unitary,
Definition \ref{def-tclass} is equivalent to \cite[Def. 2.8]{g}.

\subsection{Additivity}
\label{subsect-add}

In this subsection,
we use the notations in \textsection \ref{subsect-con}.
And we assume that $Z$ is closed.

Let $N\subseteq M$ be a hypersurface cutting $M$ into two pieces,
which we denote by $M_1',M_2'$.
Assume that $\pi\big|_N: N \rightarrow S$ is surjective and $N$ is transversal to $Z_s$ for any $s\in S$.
Then $\pi\big|_N: N \rightarrow S$ is a fibration.
Let $N \subseteq U\subseteq M$ be a tubular neighborhood such that
$\pi\big|_U: U \rightarrow S$ is isomorphic to the fibration $\pi\big|_N \circ \mathrm{pr}_2: (-1,1)\times N \rightarrow S$.
Set
\begin{equation}
M_1 = M_1' \cup \overline{U} \;,\hspace{5mm}
M_2 = M_2' \cup \overline{U} \;,\hspace{5mm}
M_3 = \overline{U} \;.
\end{equation}
For $j=1,2,3$,
we have a fibration
\begin{equation}
\pi_j := \pi\big|_{M_j}: M_j \rightarrow S \;.
\end{equation}
For $s\in S$,
we denote $Z_{j,s} = \pi_j^{-1}(s)$.
For convenience,
we will use the notations $M_0 = M$, $Z_{0,s} = Z_s$, $\pi_0 = \pi$ etc.

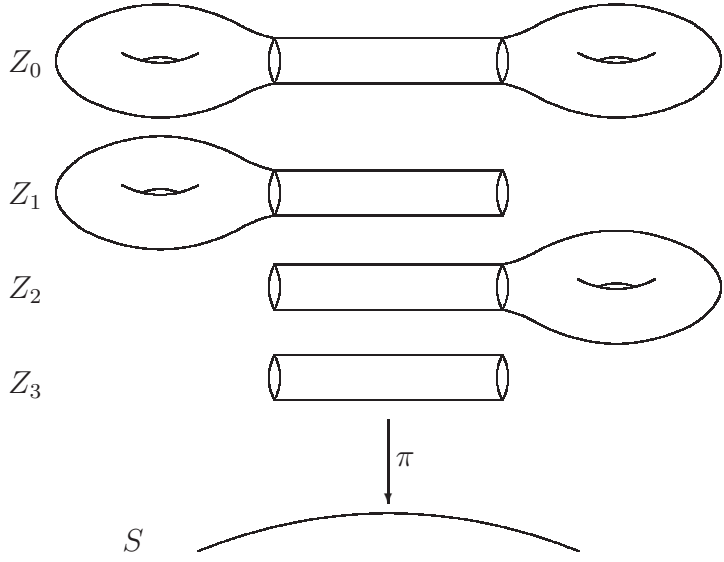
\begin{figure}[h]
\setlength{\unitlength}{5mm}
\centering
\begin{picture}(18,15)
\qbezier(1,12)(3,11)(5,12) 
\qbezier(1,14)(3,15)(5,14) 
\qbezier(1,12)(-0.5,13)(1,14) 
\qbezier(2,13.2)(3,12.7)(4,13.2) 
\qbezier(2.5,13)(3,13.2)(3.5,13) 
\qbezier(5,12)(5.5,12.3)(6,12.4) 
\qbezier(5,14)(5.5,13.7)(6,13.6) 
\qbezier(13,12)(15,11)(17,12) 
\qbezier(13,14)(15,15)(17,14) 
\qbezier(17,12)(18.5,13)(17,14) 
\qbezier(14,13.2)(15,12.7)(16,13.2) 
\qbezier(14.5,13)(15,13.2)(15.5,13) 
\qbezier(13,12)(12.5,12.3)(12,12.4) 
\qbezier(13,14)(12.5,13.7)(12,13.6) 
\qbezier(6,12.4)(9,12.4)(12,12.4) 
\qbezier(6,13.6)(9,13.6)(12,13.6) 
\qbezier(6,12.4)(5.7,13)(6,13.6) 
\qbezier[30](6,12.4)(6.3,13)(6,13.6) 
\qbezier(12,12.4)(11.7,13)(12,13.6) 
\qbezier[30](12,12.4)(12.3,13)(12,13.6) 
\qbezier(1,8.5)(3,7.5)(5,8.5) 
\qbezier(1,10.5)(3,11.5)(5,10.5) 
\qbezier(1,8.5)(-0.5,9.5)(1,10.5) 
\qbezier(2,9.7)(3,9.2)(4,9.7) 
\qbezier(2.5,9.5)(3,9.7)(3.5,9.5) 
\qbezier(5,8.5)(5.5,8.8)(6,8.9) 
\qbezier(5,10.5)(5.5,10.2)(6,10.1) 
\qbezier(6,8.9)(9,8.9)(12,8.9) 
\qbezier(6,10.1)(9,10.1)(12,10.1) 
\qbezier(6,8.9)(5.7,9.5)(6,10.1) 
\qbezier[30](6,8.9)(6.3,9.5)(6,10.1) 
\qbezier(12,8.9)(11.7,9.5)(12,10.1) 
\qbezier(12,8.9)(12.3,9.5)(12,10.1) 
\qbezier(13,6)(15,5)(17,6) 
\qbezier(13,8)(15,9)(17,8) 
\qbezier(17,6)(18.5,7)(17,8) 
\qbezier(14,7.2)(15,6.7)(16,7.2) 
\qbezier(14.5,7)(15,7.2)(15.5,7) 
\qbezier(13,6)(12.5,6.3)(12,6.4) 
\qbezier(13,8)(12.5,7.7)(12,7.6) 
\qbezier(6,6.4)(9,6.4)(12,6.4) 
\qbezier(6,7.6)(9,7.6)(12,7.6) 
\qbezier(6,6.4)(5.7,7)(6,7.6) 
\qbezier[30](6,6.4)(6.3,7)(6,7.6) 
\qbezier(12,6.4)(11.7,7)(12,7.6) 
\qbezier[30](12,6.4)(12.3,7)(12,7.6) 
\qbezier(6,4)(9,4)(12,4) 
\qbezier(6,5.2)(9,5.2)(12,5.2) 
\qbezier(6,4)(5.7,4.6)(6,5.2) 
\qbezier[30](6,4)(6.3,4.6)(6,5.2) 
\qbezier(12,4)(11.7,4.6)(12,5.2) 
\qbezier(12,4)(12.3,4.6)(12,5.2) 
\qbezier(4,0)(9,2)(14,0)
\put(-1,12.7){$Z_0$}
\put(-1,9.2){$Z_1$}
\put(-1,6.7){$Z_2$}
\put(-1,4.2){$Z_3$}
\put(2,0){$S$}
\put(9,3.5){\vector(0,-1){2.25}} \put(9.2,2.25){$\pi$}
\end{picture}
\caption{fibrations $\pi_j: M_j \rightarrow S$ with $j=0,1,2,3$}
\label{fig-gluing}
\end{figure}

For $j=0,1,2,3$,
the fiberwise cohomology $H^\bullet(Z_j,F)$ is a graded flat complex vector bundle over $S$.
We assume that each $H^k(Z_j,F)$ is filtered by flat subbundles with unitary factors.
Then
\begin{equation}
\tau^\mathrm{BL}(M_j/S,F) \in H^{\mathrm{even}\geqslant 2}(S) \hspace{2.5mm} \text{with } j=0,1,2,3
\end{equation}
are well-defined.

\begin{thm}
\label{thm-gluing}
The following identity holds,
\begin{equation}
\label{eq-thm-gluing}
\tau^\mathrm{BL}(M/S,F) =
\tau^\mathrm{BL}(M_1/S,F) + \tau^\mathrm{BL}(M_2/S,F) - \tau^\mathrm{BL}(M_3/S,F) \;.
\end{equation}
\end{thm}
\begin{proof}
For $j=0,1,2,3$,
let $T^HM_j\subseteq TM_j$ be the restriction of $T^HM$ to $M_j$,
let $g^{TZ_j}$ be the restricted metric of $g^{TZ}$ on $TZ_j$.
Let
\begin{equation}
\mathscr{T}\big(T^HM_j,g^{TZ_j},g^F\big) \in Q^S
\end{equation}
be the Bismut-Lott analytic torsion form of $\big(\pi_j: M_j \rightarrow S, F\big)$.

We consider the Mayer-Vietoris exact sequence
\begin{equation}
\label{eq-mv-es}
\cdots \rightarrow H^k(Z,F) \rightarrow H^k(Z_1,F)  \oplus H^k(Z_2,F)  \rightarrow H^k(Z_3,F)  \rightarrow \cdots \;,
\end{equation}
which is an exact sequence of flat complex vector bundles over $S$.
For $j=0,1,2,3$,
let $g^{H^\bullet(Z_j,F)}$ be the $L^2$-metric on $H^\bullet(Z_j,F)$ associated with $g^{TZ_j},g^F$.
Let
\begin{equation}
\mathscr{T}_\mathscr{H} \in Q^S
\end{equation}
be the torsion form of the exact sequence \eqref{eq-mv-es} equipped with metrics $\big(g^{H^\bullet(Z_j,F)}\big)_{j=0,1,2,3}$.

By \cite[Thm. 0.1]{pzz2},
we have
\begin{equation}
\label{eq4-pf-thm-gluing}
\mathscr{T}_\mathscr{H} + \sum_{j=0}^3 (-1)^{j(j-3)/2} \mathscr{T}\big(T^HM_j,g^{TZ_j},g^F\big) = 0
\hspace{2.5mm} \text{modulo } Q^{S,0} \;.
\end{equation}
On the other hand,
by Theorem \ref{thm-limtf} and \eqref{eq2-def-limtf},
we have
\begin{equation}
\label{eq5-pf-thm-gluing}
\mathscr{T}_\mathscr{H}^{[>0]} = \sum_{j=0}^3 (-1)^{j(j-3)/2} \mathscr{T}\big(H^\bullet_j,g^{H^\bullet_j}\big)
\hspace{2.5mm} \text{modulo } Q^{S,0} \;.
\end{equation}
By Definition \ref{def-tclass} and \eqref{eq-def-tcl},
we have
\begin{equation}
\label{eq6-pf-thm-gluing}
\tau^\mathrm{BL}(M_j/S,F) =
\Big[ \mathscr{T}^{[>0]}\big(T^HM_j,g^{TZ_j},g^F\big) + \mathscr{T}\big(H^\bullet(Z_j,F),g^{H^\bullet(Z_j,F)}\big) \Big] \;.
\end{equation}
From \eqref{eq4-pf-thm-gluing}-\eqref{eq6-pf-thm-gluing},
we obtain \eqref{eq-thm-gluing}.
This completes the proof.
\end{proof}

\subsection{Transfer and induction}
In this subsection,
we use the notations in \textsection \ref{subsect-con}.
And we assume that $Z$ is closed.

Let $\xi$ be a real vector bundle of rank $n+1$ over $M$.
Let
\begin{equation}
q: \mathbb{S}^n(\xi) \rightarrow M
\end{equation}
be the $\mathbb{S}^n$-bundle associated.
We denote by $X$ the fiber of $\pi\circ q: \mathbb{S}^n(\xi) \rightarrow S$.
We denote by $Y$ the fiber of $q: \mathbb{S}^n(\xi) \rightarrow M$.
Recall that $Z$ is the fiber of $\pi: M \rightarrow S$.
The notations are summarized in the following commutative diagram,
\begin{equation}
\xymatrix{
q^*F \ar[r] & \mathbb{S}^n(\xi) \ar[dr]^X \ar[d]_Y &  \\
F \ar[r] & M \ar[r]_Z & S \;.}
\end{equation}

The fiberwise cohomology $H^\bullet(Y,q^*F)$ is a graded flat complex vector bundle over $M$.
Since $Y\simeq\mathbb{S}^n$,
we have
\begin{equation}
\label{eqH-pf-thm-transfer}
H^0(Y,q^*F) = H^n(Y,q^*F) = F \;,\hspace{5mm}
H^k(Y,q^*F) = 0 \hspace{2.5mm} \text{for } k\neq 0,n \;.
\end{equation}
The fiberwise cohomologies $H^\bullet\big(Z,H^\bullet(Y,q^*F)\big)$ and $H^\bullet(X,q^*F)$ are graded flat complex vector bundles over $S$.
Moreover,
we have
\begin{align}
\label{eqH2-pf-thm-transfer}
\begin{split}
H^k(X,q^*F)
& = H^k\big(Z,H^0(Y,q^*F)\big) \oplus H^{k-n}\big(Z,H^n(Y,q^*F)\big) \\
& = H^k(Z,F) \oplus H^{k-n}(Z,F) \;.
\end{split}
\end{align}
From \eqref{eqH-pf-thm-transfer} and \eqref{eqH2-pf-thm-transfer},
we see that both $H^\bullet(Y,q^*F)$ and $H^\bullet(X,q^*F)$ are filtered by flat subbundles with unitary factors.
Hence
\begin{equation}
\tau^\mathrm{BL}\big(\mathbb{S}^n(\xi)/M,q^*F\big) \in H^{\mathrm{even}\geqslant 2}(M) \;,\hspace{5mm}
\tau^\mathrm{BL}\big(\mathbb{S}^n(\xi)/S,q^*F\big) \in H^{\mathrm{even}\geqslant 2}(S)
\end{equation}
are well-defined.

The following theorem is a direct consequence of \cite[Thm. 0.1]{ma}.

\begin{thm}
\label{thm-transfer}
The following identity holds,
\begin{equation}
\label{eq-thm-transfer}
\tau^\mathrm{BL}\big(\mathbb{S}^n(\xi)/S,q^*F\big) =
\chi(\mathbb{S}^n) \tau^\mathrm{BL}(M/S,F) + \int_Z e(TZ) \tau^\mathrm{BL}\big(\mathbb{S}^n(\xi)/M,q^*F\big) \;.
\end{equation}
\end{thm}
\begin{proof}
Let $g^{TX},g^{TY}$ be Riemannian metrics on $TX,TY$.
Recall that $g^{TZ}$ is a Riemannian metric on $TZ$.
Recall that $g^F$ is a flat Hermitian metric on $F$.
\begin{itemize}
\item[-] Let $g^{H^\bullet(X,q^*F)}$ be the $L^2$-metric on $H^\bullet(X,q^*F)$ associated with $g^{TX}$, $q^*g^F$.
\item[-] Let $g^{H^\bullet(Y,q^*F)}$ be the $L^2$-metric on $H^\bullet(Y,q^*F)$ associated with $g^{TY}$, $q^*g^F$.
\item[-] Let $g^{H^\bullet(Z,H^\bullet(Y,q^*F))}$ be the $L^2$-metric on $H^\bullet\big(Z,H^\bullet(Y,q^*F)\big)$ associated with $g^{TZ}$, $g^{H^\bullet(Y,q^*F)}$.
\item[-] Recall that $g^{H^\bullet(Z,F)}$ is the $L^2$-metric on $H^\bullet(Z,F)$ associated with $g^{TZ}$, $g^F$.
\end{itemize}
We assume that the volume of any fiber $Y$ with respect to $g^{TY}$ equals $1$.
Then,
under the isomorphism \eqref{eqH-pf-thm-transfer},
we have
\begin{equation}
\label{eqH3-pf-thm-transfer}
g^{H^0(Y,q^*F)} = g^{H^n(Y,q^*F)} = g^F
\end{equation}
and
\begin{equation}
\label{eqH4-pf-thm-transfer}
g^{H^\bullet(Z,H^0(Y,q^*F))} = g^{H^\bullet(Z,H^n(Y,q^*F))} = g^{H^\bullet(Z,F)} \;.
\end{equation}

Under the isomorphism \eqref{eqH2-pf-thm-transfer},
we may view $g^{H^\bullet(Z,H^0(Y,q^*F))} \oplus g^{H^{\bullet-n}(Z,H^n(Y,q^*F))}$ as a metric on $H^\bullet(X,q^*F)$.
Recall that $\mathscr{T}(\cdot,\cdot)$ was defined in the paragraph containing \eqref{eq-EE}.
For a graded flat complex vector bundle $(E^\bullet,\n^{E^\bullet})$ over $S$
and Hermitian metrics $g^{E^\bullet} = \bigoplus_k g^{E^k},{g^{E^\bullet}}' = \bigoplus_k {g^{E^k}}'$ on $E^\bullet$,
we denote
\begin{equation}
\label{eq-EE-grad}
\mathscr{T}\big(g^{E^\bullet},{g^{E^\bullet}}'\big) = \sum_k (-1)^k \mathscr{T}\big(g^{E^k},{g^{E^k}}'\big) \;.
\end{equation}
By Theorem \ref{thm-limtf}, \eqref{eq2-def-limtf}, \eqref{eqH4-pf-thm-transfer} and \eqref{eq-EE-grad},
we have
\begin{align}
\label{eqtf-pf-thm-transfer}
\begin{split}
& \mathscr{T}^{[>0]}\Big(g^{H^\bullet(Z,H^0(Y,q^*F))} \oplus g^{H^{\bullet-n}(Z,H^n(Y,q^*F))},g^{H^\bullet(X,q^*F)}\Big) \\
& = \big(1 + (-1)^n\big) \mathscr{T}\big(H^\bullet(Z,F),g^{H^\bullet(Z,F)}\big) - \mathscr{T}\big(H^\bullet(X,q^*F),g^{H^\bullet(X,q^*F)}\big) \\
& = \chi(\mathbb{S}^n)\mathscr{T}\big(H^\bullet(Z,F),g^{H^\bullet(Z,F)}\big) - \mathscr{T}\big(H^\bullet(X,q^*F),g^{H^\bullet(X,q^*F)}\big)
\hspace{2.5mm} \text{modulo } Q^{S,0} \;.
\end{split}
\end{align}

Recall that $T^HM \subseteq TM$ is a complement of $TZ$.
Let
\begin{equation}
\mathscr{T}\big(T^HM,g^{TZ},g^{H^\bullet(Y,q^*F)}\big) := \sum_{k=0}^n (-1)^k \mathscr{T}\big(T^HM,g^{TZ},g^{H^k(Y,q^*F)}\big) \in Q^S/Q^{S,0}
\end{equation}
be the Bismut-Lott analytic torsion form of $\big(\pi: M \rightarrow S, H^\bullet(Y,q^*F)\big)$.
Let $T^H\mathbb{S}^n(\xi) \subseteq T\mathbb{S}^n(\xi)$ be a complement of $TY$.
Let
\begin{equation}
\mathscr{T}\big(T^H\mathbb{S}^n(\xi),g^{TY},q^*g^F\big) \in Q^M/Q^{M,0}
\end{equation}
be the Bismut-Lott analytic torsion form of $\big(q: \mathbb{S}^n(\xi) \rightarrow M, q^*F\big)$.
Let $T^H_\mathrm{tot}\mathbb{S}^n(\xi) \subseteq T\mathbb{S}^n(\xi)$ be a complement of $TX$.
Let
\begin{equation}
\mathscr{T}\big(T^H_\mathrm{tot}\mathbb{S}^n(\xi),g^{TX},q^*g^F\big) \in Q^S/Q^{S,0}
\end{equation}
be the Bismut-Lott analytic torsion form of $\big(\pi \circ q: \mathbb{S}^n(\xi) \rightarrow S, q^*F\big)$.
By \cite[Thm. 0.1]{ma},
we have
\begin{align}
\label{eqM-pf-thm-transfer}
\begin{split}
& \mathscr{T}\big(T^H_\mathrm{tot}\mathbb{S}^n(\xi),g^{TX},q^*g^F\big) \\
& = \mathscr{T}\big(T^HM,g^{TZ},g^{H^\bullet(Y,q^*F)}\big) + \int_Z e(TZ) \mathscr{T}\big(T^H\mathbb{S}^n(\xi),g^{TY},q^*g^F\big) \\
& \hspace{5mm} + \mathscr{T}\Big(g^{H^\bullet(Z,H^0(Y,q^*F))} \oplus g^{H^{\bullet-n}(Z,H^n(Y,q^*F))},g^{H^\bullet(X,q^*F)}\Big)
\hspace{2.5mm} \text{modulo } Q^{S,0} \;.
\end{split}
\end{align}
Here we remark that $\mathscr{T}\big(T^H\mathbb{S}^n(\xi),g^{TY},q^*g^F\big)$ is closed.
This follows from \eqref{eq-datf} with $(\pi: M\rightarrow S, F)$ replaced by $\big(q: \mathbb{S}^n(\xi) \rightarrow M, q^*F\big)$.

By \eqref{eqH3-pf-thm-transfer} and the fact that $g^F$ is a flat Hermitian metric,
we have
\begin{equation}
\label{eqT1a-pf-thm-transfer}
\mathscr{T}\big(H^\bullet(Y,q^*F),g^{H^\bullet(Y,q^*F)}\big) = 0 \;.
\end{equation}
By Definition \ref{def-tclass}, \eqref{eq-def-tcl} and \eqref{eqT1a-pf-thm-transfer},
we have
\begin{equation}
\label{eqT1-pf-thm-transfer}
\tau^\mathrm{BL}\big(\mathbb{S}^n(\xi)/M,q^*F\big)
= \Big[ \mathscr{T}^{[>0]}\big(T^H\mathbb{S}^n(\xi),g^{TY},q^*g^F\big) \Big] \;.
\end{equation}
By Definition \ref{def-tclass} and \eqref{eq-def-tcl},
we have
\begin{align}
\label{eqT2-pf-thm-transfer}
\begin{split}
& \tau^\mathrm{BL}\big(\mathbb{S}^n(\xi)/S,q^*F\big) \\
& = \Big[ \mathscr{T}^{[>0]}\big(T^H_\mathrm{tot}\mathbb{S}^n(\xi),g^{TX},q^*g^F\big) + \mathscr{T}\big(H^\bullet(X,q^*F),g^{H^\bullet(X,q^*F)}\big) \Big] \;.
\end{split}
\end{align}
On the other hand,
by \eqref{eqH-pf-thm-transfer} and \eqref{eqH3-pf-thm-transfer},
we have
\begin{align}
\label{eqT3-pf-thm-transfer}
\begin{split}
& \mathscr{T}\big(T^HM,g^{TZ},g^{H^\bullet(Y,q^*F)}\big) \\
& = \big(1 + (-1)^n\big) \mathscr{T}\big(T^HM,g^{TZ},g^F\big)
= \chi(\mathbb{S}^n)\mathscr{T}\big(T^HM,g^{TZ},g^F\big) \;.
\end{split}
\end{align}
From Definition \ref{def-tclass}, \eqref{eq-def-tcl}, \eqref{eqtf-pf-thm-transfer} and \eqref{eqM-pf-thm-transfer}-\eqref{eqT3-pf-thm-transfer},
we obtain \eqref{eq-thm-transfer}.
This completes the proof.
\end{proof}

Now we assume that there is a finite covering
\begin{equation}
p: M \rightarrow M_\sim
\end{equation}
and a fibration
\begin{equation}
\pi_\sim: M_\sim \rightarrow S
\end{equation}
such that $\pi = \pi_\sim \circ p$.
Let $Z_\sim$ be the fiber of $\pi_\sim: M_\sim \rightarrow S$.
Let $\pi_*F$ be the direct image of $F$,
which is flat complex vector bundle over $M_\sim$.
The notations are summarized in the following commutative diagram,
\begin{equation}
\xymatrix{
F \ar[r] & M \ar[dr]^Z \ar[d] &  \\
p_*F \ar[r] & M_\sim \ar[r]_{Z_\sim} & S \;.}
\end{equation}

Let $g^{p_*F}$ be the Hermitian metric on $g^{p_*F}$ induced by $g^F$.
Since $g^F$ is a flat Hermitian metric,
so is $g^{p_*F}$.
The fiberwise cohomology $H^\bullet(Z_\sim,p_*F)$ is a graded flat complex vector bundle over $S$.
Moreover,
we have
\begin{equation}
\label{eqH-pf-thm-induction}
H^\bullet(Z,F) = H^\bullet(Z_\sim,p_*F) \;.
\end{equation}
Since $H^k(Z,F)$ is filtered by flat subbundles with unitary factors,
so is $H^k(Z_\sim,p_*F)$.
Hence
\begin{equation}
\tau^\mathrm{BL}\big(M_\sim/S,p_*F\big) \in H^{\mathrm{even}\geqslant 2}(S)
\end{equation}
is well-defined.

The following theorem is a direct consequence of \cite[Thm. 0.1]{ma}.

\begin{thm}
\label{thm-induction}
The following identity holds,
\begin{equation}
\label{eq-thm-induction}
\tau^\mathrm{BL}(M/S,F) =
\tau^\mathrm{BL}\big(M_\sim/S,p_*F\big) \;.
\end{equation}
\end{thm}
\begin{proof}
Let $g^{TZ_\sim}$ be a Riemannian metric on $TZ_\sim$.
Let $g^{H^\bullet(Z_\sim,p_*F)}$ be the $L^2$-metric on $H^\bullet(Z_\sim,p_*F)$ associated with $g^{TZ_\sim}, g^{p_*F}$.
Under the isomorphism \eqref{eqH-pf-thm-induction},
we may view $g^{H^\bullet(Z_\sim,p_*F)}$ as a metric on $g^{H^\bullet(Z,F)}$.
By Theorem \ref{thm-limtf}, \eqref{eq2-def-limtf} and \eqref{eq-EE-grad},
we have
\begin{align}
\label{eq1-pf-thm-induction}
\begin{split}
& \mathscr{T}^{[>0]}\big(g^{H^\bullet(Z_\sim,p_*F)},g^{H^\bullet(Z,F)}\big) \\
& = \mathscr{T}\big(H^\bullet(Z_\sim,p_*F),g^{H^\bullet(Z_\sim,p_*F)}\big) - \mathscr{T}\big(H^\bullet(Z,F),g^{H^\bullet(Z,F)}\big)
\hspace{2.5mm} \text{modulo } Q^{S,0} \;.
\end{split}
\end{align}

Let $T^HM_\sim \subseteq TM_\sim$ be a complement of $TZ_\sim$.
Let
\begin{equation}
\mathscr{T}^{[>0]}\big(T^HM_\sim,g^{TZ_\sim},g^{p_*F}\big) \in Q^S/Q^{S,0}
\end{equation}
be the Bismut-Lott analytic torsion form of $(\pi_\sim: M_\sim \rightarrow S, p_*F)$.
Since $g^F$ is a flat Hermitian metric and $p: M \rightarrow M_\sim$ is a finite covering,
the Bismut-Lott analytic torsion form of $(p: M \rightarrow M_\sim, F)$ vanishes.
Then,
by \cite[Thm. 0.1]{ma},
we have
\begin{align}
\label{eqM-pf-thm-induction}
\begin{split}
& \mathscr{T}\big(T^HM,g^{TZ},g^F\big) \\
& = \mathscr{T}\big(T^HM_\sim,g^{TZ_\sim},g^{p_*F}\big)
+ \mathscr{T}\big(g^{H^\bullet(Z_\sim,p_*F)},g^{H^\bullet(Z,F)}\big)
\hspace{2.5mm} \text{modulo } Q^{S,0} \;.
\end{split}
\end{align}

By Definition \ref{def-tclass} and \eqref{eq-def-tcl},
we have
\begin{align}
\label{eq2-pf-thm-induction}
\begin{split}
& \tau^\mathrm{BL}\big(M_\sim/S,p_*F\big) \\
& = \Big[ \mathscr{T}^{[>0]}\big(T^HM_\sim,g^{TZ_\sim},g^{p_*F}\big) + \mathscr{T}\big(H^\bullet(Z_\sim,p_*F),g^{H^\bullet(Z_\sim,p_*F)}\big) \Big] \;.
\end{split}
\end{align}
From Definition \ref{def-tclass}, \eqref{eq-def-tcl}, \eqref{eq1-pf-thm-induction}, \eqref{eqM-pf-thm-induction} and \eqref{eq2-pf-thm-induction},
we obtain \eqref{eq-thm-induction}.
This completes the proof.
\end{proof}

\subsection{Triviality}
In this subsection,
we use the notations in \textsection \ref{subsect-con}.

For a positive integer $k$,
let
\begin{equation}
\tau^\mathrm{BL}_k(M/S,F) \in H^{2k}(S)
\end{equation}
be the component of $\tau^\mathrm{BL}(M/S,F)$ of degree $2k$.

The following theorem is due to Bismut and Lott \cite{bl}.

\begin{thm}
\label{thm-trivial}
If $(F,\n^F)$ is a trivial flat complex line bundle over $M$,
then
\begin{equation}
\tau^\mathrm{BL}_{2k+1}(M/S,F) = 0 \hspace{2.5mm} \text{for } k\in\N \;.
\end{equation}
\end{thm}
\begin{proof}
This is a direct consequence of Theorems \ref{thm-limtf-4kplus2}, \ref{thm-atf-4kplus2}, Definition \ref{def-tclass} and \eqref{eq-def-tcl}.
\end{proof}

\subsection{Circle bundle}
\label{subsect-circle}

Let $n$ be a positive integer.
Let $L$ be a complex line bundle over $S$.
Let $L^{\otimes n}$ be $n$-th tensor power of $L$.
Let
\begin{equation}
\pi_n: \mathbb{S}^1(L^{\otimes n}) \rightarrow S
\end{equation}
be the $\mathbb{S}^1$-bundle associated with $L^{\otimes n}$.
Let
\begin{equation}
q_n: \mathbb{S}^1(L) \rightarrow \mathbb{S}^1(L^{\otimes n})
\end{equation}
be the $n$-covering induced by $L \rightarrow L^{\otimes n}, s \mapsto s^{\otimes n}$.
Let $\alpha\in\C^*$ such that $\alpha^n=1$.
Let $\big(F_{\alpha},\n^{F_\alpha}\big)$ be the unique flat complex line bundle over $\mathbb{S}^1(L^{\otimes n})$ such that
\begin{itemize}
\item[-] The pull-back $q_n^*\big(F_{\alpha},\n^{F_\alpha}\big)$ is a trivial flat complex line bundle over $\mathbb{S}^1(L)$.
\item[-] The holonomy of $\big(F_{\alpha},\n^{F_\alpha}\big)$ along the fiber of $\pi_n: \mathbb{S}^1(L^{\otimes n}) \rightarrow S$ equals $\alpha$.
\end{itemize}

For $k$ an integer greater than $1$,
let $\mathrm{Li}_k$ be the polylogarithm function
\begin{align}
\begin{split}
\mathrm{Li}_k:  \big\{ z\in\C \;:\; |z| \leqslant 1 \big\} & \rightarrow \C \\
z & \mapsto \sum_{m=1}^\infty m^{-k}z^m \;.
\end{split}
\end{align}

Let $\omega \in H^2(S)$ be the first Chern class of $L^{\otimes n}$.

The following theorem is due to Bismut and Lott \cite{bl}.

\begin{thm}
\label{thm-circle}
The following identity holds,
\begin{align}
\label{eq-thm-circle}
\begin{split}
\tau^\mathrm{BL}\big(\mathbb{S}^1(L^{\otimes n})/S,F_\alpha\big)
& = \sum_{k>0 \; \mathrm{even}} (-1)^{k/2} \frac{(2k+1)!}{2^{2k}(k!)^2 (2\pi)^k} \mathrm{Re}\big(\mathrm{Li}_{k+1}(\alpha)\big) \omega^k \\
& \hspace{2.5mm} + \sum_{k>0 \; \mathrm{odd}} (-1)^{(k-1)/2} \frac{(2k+1)!}{2^{2k}(k!)^2(2\pi)^k} \mathrm{Im}\big(\mathrm{Li}_{k+1}(\alpha)\big) \omega^k \;.
\end{split}
\end{align}
\end{thm}
\begin{proof}
For $\alpha \neq 1$,
the identity \eqref{eq-thm-circle} is a direct consequence of \cite[Cor. 4.14]{bl}.
Since \eqref{eq-thm-circle} concerns the analytic torsion forms of positive degree,
the proof of \cite[Cor. 4.14]{bl} equally implies \eqref{eq-thm-circle} for $\alpha = 1$.
\end{proof}

The following theorem is due to Igusa and Klein \cite{ig4} (see also \cite[Thm. 3.1]{ig3}).

\begin{thm}
\label{thm2-circle}
The following identity holds,
\begin{align}
\begin{split}
\tau^\mathrm{IK}\big(\mathbb{S}^1(L^{\otimes n})/S,F_\alpha\big)
& = \sum_{k \; \mathrm{even}} (-1)^{(k+2)/2} \frac{1}{k!} \mathrm{Re}\big(\mathrm{Li}_{k+1}(\alpha)\big) \omega^k \\
& \hspace{2.5mm} + \sum_{k \; \mathrm{odd}} (-1)^{(k+1)/2} \frac{1}{k!} \mathrm{Im}\big(\mathrm{Li}_{k+1}(\alpha)\big) \omega^k \;.
\end{split}
\end{align}
\end{thm}

\subsection{Proofs of Theorems \ref{thm-bg}', \ref{thm-bg2}}

\begin{proof}[Proof of Theorem \ref{thm-bg}']
Recall that Axioms 1-7 were stated in \textsection \ref{subsect-axiom}.
Axiom 1 trivially holds for $\tau^\mathrm{BL}$.
By Theorem \ref{thm-gluing},
Axiom 2 holds for $\tau^\mathrm{BL}$.
By Theorem \ref{thm-transfer},
Axiom 3 holds for $\tau^\mathrm{BL}$.
Hence $\tau^\mathrm{BL}$ satisfies Igusa's axiomatization \cite[\textsection 3]{ig2}.

By \cite[Cor. 4.5]{ig2},
there exists $a_{2k}\in\R$ such that
for any $M/S$ with odd-dimensional fiber,
we have
\begin{equation}
\frac{2^{4k}\big((2k)!\big)^2}{(4k+1)!} \tau^\mathrm{BL}_{2k}(M/S,\mathbb{1}) = a_{2k} \tau^\mathrm{IK}_{2k}(M/S,\mathbb{1}) \;.
\end{equation}
Now let $M/S$ be a circle bundle.
From Theorems \ref{thm-circle}, \ref{thm2-circle},
we get
\begin{equation}
a_{2k} = - \frac{(2k)!}{(2\pi)^{2k}} \;.
\end{equation}

By \cite[Cor. 4.5]{ig2},
there exists $b_{2k}\in\R$ such that
for any $M/S$ with even-dimensional fiber $Z$,
we have
\begin{equation}
\tau^\mathrm{BL}_{2k}(M/S,\mathbb{1}) = b_{2k} \Big[\int_Z e(TZ)\mathrm{ch}(TZ) \Big]^{[4k]} \in H^{4k}(S) \;.
\end{equation}
Now let $M/S$ be a $\mathbb{S}^2$-bundle.
Let $TZ$ be a Riemannian metric on $TZ$ such that the volume of any fiber of $M/S$ equals $1$.
Let $g^\mathbb{1}$ be the canonical metric on $\mathbb{1}$.
Then the $L^2$-metric on $H^\bullet(Z)$ is flat.
By Definition \ref{def-tclass},
we have
\begin{equation}
\tau^\mathrm{BL}_{2k}(M/S,\mathbb{1}) = \Big[ \mathscr{T}\big(T^HM,g^{TZ},g^\mathbb{1}\big) \Big]^{[4k]} \in H^{4k}(S) \;.
\end{equation}
On the other hand,
by \cite[Thm. 3.26]{bl},
we have
\begin{equation}
\mathscr{T}\big(T^HM,g^{TZ},g^\mathbb{1}\big) = 0 \;.
\end{equation}
Hence $\tau^\mathrm{BL}_{2k}(M/S,\mathbb{1}) = 0$ and $b_{2k} = 0$.
This completes the proof.
\end{proof}

\begin{proof}[Proof of Theorem \ref{thm-bg2}]
In the proof of Theorem \ref{thm-bg}',
we showed that Axioms 1-3 hold for $\tau^\mathrm{BL}$.
By Theorem \ref{thm-trivial},
Axiom 4 holds for $\tau^\mathrm{BL}$.
Axiom 5 trivially holds for $\tau^\mathrm{BL}$.
By Theorem \ref{thm-induction},
Axiom 6 holds for $\tau^\mathrm{BL}$.
By Theorem \ref{thm-circle},
Axiom 7 holds for $\tau^\mathrm{BL}$.
Hence $\tau^\mathrm{BL}$ satisfies Ohrt's axiomatization \cite[\textsection 2.2]{ohrt}.

By \cite[Theorem 0.1]{ohrt},
there exist $a_{2k},b_{2k}\in\R$ such that for any $(M/S,F)$ under consideration,
we have
\begin{equation}
\label{eq0-pf-thm-bg2}
\frac{2^{4k}\big((2k)!\big)^2}{(4k+1)!} \tau^\mathrm{BL}_{2k}(M/S,F)
= a_{2k} \tau^\mathrm{IK}_{2k}(M/S,F) + b_{2k} \mathrm{rk}F \Big[\int_Z e(TZ)\mathrm{ch}(TZ)\Big]^{[4k]} \;.
\end{equation}
We view \eqref{eq0-pf-thm-bg2} as a system of equations of $a_{2k},b_{2k}$.
Taking $F=\mathbb{1}$,
we get
\begin{equation}
\label{eq1-pf-thm-bg2}
\frac{2^{4k}\big((2k)!\big)^2}{(4k+1)!} \tau^\mathrm{BL}_{2k}(M/S,\mathbb{1})
= a_{2k} \tau^\mathrm{IK}_{2k}(M/S,\mathbb{1}) + b_{2k} \Big[\int_Z e(TZ)\mathrm{ch}(TZ)\Big]^{[4k]} \;.
\end{equation}
Comparing Theorem \ref{thm-bg} with \eqref{eq1-pf-thm-bg2},
we know that
\begin{equation}
\label{eq2-pf-thm-bg2}
a_{2k} = - \frac{(2k)!}{(2\pi)^{2k}} \;,\hspace{5mm} b_{2k} = \frac{\zeta'(-2k)}{2}
\end{equation}
is a solution of \eqref{eq1-pf-thm-bg2}.
On the other hand,
by \cite[Prop. 4.6, 4.7]{ig2},
the solution of \eqref{eq1-pf-thm-bg2} is unique.
Hence \eqref{eq2-pf-thm-bg2} is the unique solution of \eqref{eq0-pf-thm-bg2}.

By \cite[Theorem 0.1]{ohrt},
there exist $c_{2k+1}\in\R$ such that for any $(M/S,F)$ under consideration,
we have
\begin{equation}
\frac{2^{4k+2}\big((2k+1)!\big)^2}{(4k+3)!} \tau^\mathrm{BL}_{2k+1}(M/S,F)
= c_{2k+1} \tau^\mathrm{IK}_{2k+1}(M/S,F) \;.
\end{equation}
Now let $M/S$ be a circle bundle.
From Theorems \ref{thm-circle}, \ref{thm2-circle},
we get
\begin{equation}
c_{2k+1} = - \frac{(2k+1)!}{(2\pi)^{2k+1}} \;.
\end{equation}
This completes the proof.
\end{proof}

\input{torsionclassbib}


\end{document}

%% file: torsionclassbib.tex
\def\cprime{$'$}
\providecommand{\bysame}{\leavevmode\hbox to3em{\hrulefill}\thinspace}
\providecommand{\MR}{\relax\ifhmode\unskip\space\fi MR }
\providecommand{\MRhref}[2]{%
  \href{http://www.ams.org/mathscinet-getitem?mr=#1}{#2}
}
\providecommand{\href}[2]{#2}

%% file: torsionclass.bbl
\begin{thebibliography}{10}

\bibitem{bdkw}
B.~Badzioch, W.~Dorabia{\l}a, J.~R. Klein, and B.~Williams, \emph{Equivalence
  of higher torsion invariants}, Adv. Math. \textbf{226} (2011), no.~3,
  2192--2232.

\bibitem{bgv}
N.~Berline, Getzler., and M.~Vergne, \emph{Heat kernels and {D}irac operators},
  Grundlehren Text Editions, Springer-Verlag, Berlin, 2004, Corrected reprint
  of the 1992 original.

\bibitem{b}
J.-M. Bismut, \emph{The {A}tiyah-{S}inger index theorem for families of {D}irac
  operators: two heat equation proofs}, Invent. Math. \textbf{83} (1986),
  no.~1, 91--151.

\bibitem{b97}
J.-M. Bismut, \emph{Holomorphic families of immersions and higher analytic torsion
  forms}, Ast\'{e}risque (1997), no.~244, viii+275.

\bibitem{bg}
J.-M. Bismut and S.~Goette, \emph{Families torsion and {M}orse functions},
  Ast\'erisque (2001), no.~275, x+293 pp.

\bibitem{bg2}
J.-M. Bismut and S.~Goette, \emph{Equivariant de {R}ham torsions}, Ann. of Math. (2) \textbf{159}
  (2004), no.~1, 53--216.

\bibitem{bl}
J.-M. Bismut and J.~Lott, \emph{Flat vector bundles, direct images and higher
  real analytic torsion}, J. Amer. Math. Soc. \textbf{8} (1995), no.~2,
  291--363.

\bibitem{bz}
J.-M. Bismut and W.~Zhang, \emph{An extension of a theorem by {C}heeger and
  {M}\"uller}, Ast\'erisque (1992), no.~205, 235 pp, With an appendix by
  Fran{\c{c}}ois Laudenbach.

\bibitem{bz2}
J.-M. Bismut and W.~Zhang, \emph{Milnor and {R}ay-{S}inger metrics on the equivariant determinant
  of a flat vector bundle}, Geom. Funct. Anal. \textbf{4} (1994), no.~2,
  136--212.

\bibitem{bm}
J.~Br{\"u}ning and X.~Ma, \emph{On the gluing formula for the analytic
  torsion}, Math. Z. \textbf{273} (2013), no.~3-4, 1085--1117.

\bibitem{bu}
U.~Bunke, \emph{Equivariant higher analytic torsion and equivariant {E}uler
  characteristic}, Amer. J. Math. \textbf{122} (2000), no.~2, 377--401.

\bibitem{c-cm}
J.~Cheeger, \emph{Analytic torsion and the heat equation}, Ann. of Math. (2)
  \textbf{109} (1979), no.~2, 259--322.

\bibitem{dr}
G.~de~Rham, \emph{Complexes \`a automorphismes et hom\'{e}omorphie
  diff\'{e}rentiable}, Ann. Inst. Fourier Grenoble \textbf{2} (1950), 51--67
  (1951).

\bibitem{dww}
W.~Dwyer, M.~Weiss, and B.~Williams, \emph{A parametrized index theorem for the
  algebraic {$K$}-theory {E}uler class}, Acta Math. \textbf{190} (2003), no.~1,
  1--104.

\bibitem{fr}
W.~Franz, \emph{\"{U}ber die {T}orsion einer \"{U}berdeckung}, J. Reine Angew.
  Math. \textbf{173} (1935), 245--254.

\bibitem{g}
S.~Goette, \emph{Torsion invariants for families}, Ast\'erisque (2009), no.~328,
  161--206 (2010).

\bibitem{g-a1}
S.~Goette, \emph{Morse theory and higher torsion invariants {I}},
  arXiv:math/0111222.

\bibitem{g-a2}
S.~Goette, \emph{Morse theory and higher torsion invariants {II}},
  arXiv:math/0305287.

\bibitem{g-i}
S.~Goette and K.~Igusa, \emph{Exotic smooth structures on topological fiber
  bundles {II}}, Trans. Amer. Math. Soc. \textbf{366} (2014), no.~2, 791--832.

\bibitem{g-i-w}
S.~Goette, K.~Igusa, and B.~Williams, \emph{Exotic smooth structures on
  topological fiber bundles {I}}, Trans. Amer. Math. Soc. \textbf{366} (2014),
  no.~2, 749--790.

\bibitem{ig}
K.~Igusa, \emph{Higher {F}ranz-{R}eidemeister torsion}, AMS/IP Studies in
  Advanced Mathematics, vol.~31, American Mathematical Society, Providence, RI;
  International Press, Somerville, MA, 2002.

\bibitem{ig2}
K.~Igusa, \emph{Axioms for higher torsion invariants of smooth bundles}, J.
  Topol. \textbf{1} (2008), no.~1, 159--186.

\bibitem{ig3}
K.~Igusa, \emph{Outline of higher {I}gusa-{K}lein torsion},
  \url{http://people.brandeis.edu/~igusa/Papers/AIM0911.pdf}.
  
\bibitem{ig4}
K.~Igusa and J.~Klein, \emph{The {B}orel regulator map on pictures. {II}. {A}n
  example from {M}orse theory}, $K$-Theory \textbf{7} (1993), no.~3, 225--267.
  


\bibitem{lr}
J.~Lott and M.~Rothenberg, \emph{Analytic torsion for group actions}, J.
  Differential Geom. \textbf{34} (1991), no.~2, 431--481.

\bibitem{lu}
W.~L{\"u}ck, \emph{Analytic and topological torsion for manifolds with boundary
  and symmetry}, J. Differential Geom. \textbf{37} (1993), no.~2, 263--322.

\bibitem{ma}
X.~Ma, \emph{Functoriality of real analytic torsion forms}, Israel J. Math.
  \textbf{131} (2002), 1--50.

\bibitem{miller}
E.~Miller, \emph{The homology of the mapping class group}, J. Differential
  Geom. \textbf{24} (1986), no.~1, 1--14.

\bibitem{mil}
J.~Milnor, \emph{Whitehead torsion}, Bull. Amer. Math. Soc. \textbf{72} (1966),
  358--426.

\bibitem{mor}
S.~Morita, \emph{Characteristic classes of surface bundles}, Bull. Amer. Math.
  Soc. (N.S.) \textbf{11} (1984), no.~2, 386--388.

\bibitem{m-cm}
W.~M{\"u}ller, \emph{Analytic torsion and {$R$}-torsion of {R}iemannian
  manifolds}, Adv. in Math. \textbf{28} (1978), no.~3, 233--305.

\bibitem{m-cm-2}
W.~M{\"u}ller, \emph{Analytic torsion and {$R$}-torsion for unimodular
  representations}, J. Amer. Math. Soc. \textbf{6} (1993), no.~3, 721--753.

\bibitem{mum}
D.~Mumford, \emph{Towards an enumerative geometry of the moduli space of
  curves}, Arithmetic and geometry, {V}ol. {II}, Progr. Math., vol.~36,
  Birkh\"{a}user Boston, Boston, MA, 1983, pp.~271--328.

\bibitem{ohrt}
C.~Ohrt, \emph{Axioms for higher twisted torsion invariants of smooth bundles},
  Algebr. Geom. Topol. \textbf{17} (2017), no.~6, 3665--3701.

\bibitem{pzz}
M.~Puchol, Y.~Zhang, and J.~Zhu, \emph{Scattering matrices and analytic torsions}, Anal. PDE
  \textbf{14} (2021), no.~1, 77--134.

\bibitem{pzz2}
M.~Puchol, Y.~Zhang, and J.~Zhu, \emph{Adiabatic limit, {W}itten deformation
  and analytic torsion forms}, arXiv:2009.13925, 76 pages.

\bibitem{rs}
D.~B. Ray and I.~M. Singer, \emph{{$R$}-torsion and the {L}aplacian on
  {R}iemannian manifolds}, Adv. Math. \textbf{7} (1971), 145--210.

\bibitem{rei}
K.~Reidemeister, \emph{Homotopieringe und {L}insenr\"aume}, Abh. Math. Sem.
  Univ. Hamburg \textbf{11} (1935), no.~1, 102--109.

\bibitem{sm}
S.~Smale, \emph{On gradient dynamical systems}, Ann. of Math. (2) \textbf{74}
  (1961), 199--206.

\bibitem{vi}
S.~M. Vishik, \emph{Generalized {R}ay-{S}inger conjecture. {I}. {A} manifold
  with a smooth boundary}, Comm. Math. Phys. \textbf{167} (1995), no.~1,
  1--102.

\bibitem{wa}
J.~B. Wagoner, \emph{Diffeomorphisms, {$K_{2}$}, and analytic torsion},
  Algebraic and geometric topology ({P}roc. {S}ympos. {P}ure {M}ath.,
  {S}tanford {U}niv., {S}tanford, {C}alif., 1976), {P}art 1, Proc. Sympos. Pure
  Math., XXXII, Amer. Math. Soc., Providence, R.I., 1978, pp.~23--33.

\bibitem{wh}
J.~H.~C. Whitehead, \emph{Simple homotopy types}, Amer. J. Math. \textbf{72}
  (1950), 1--57.

\bibitem{imrn-zhu}
J.~Zhu, \emph{On the gluing formula of real analytic torsion forms}, Int. Math.
  Res. Not. IMRN (2015), no.~16, 6793--6841.

\bibitem{israel-zhu}
J.~Zhu, \emph{Gluing formula of real analytic torsion forms and adiabatic
  limit}, Israel J. Math. \textbf{215} (2016), no.~1, 181--254.

\end{thebibliography}
